\documentclass[a4paper,12pt]{article}

\def\drafts{0}

\usepackage{draftwatermark}

\if\drafts1
\SetWatermarkText{DRAFT }
\SetWatermarkScale{3}
\else
\SetWatermarkText{}
\fi
\usepackage{array}
 \usepackage[T1]{fontenc}  
  \usepackage{cancel,soul}
   \usepackage{tikz}
   
\usetikzlibrary{calc}\usetikzlibrary{decorations.pathreplacing} 

\usepackage{amsmath,amsfonts,amssymb,amsthm,mathrsfs}
\usepackage{graphicx,xcolor}
\usepackage{stmaryrd}
\usepackage{scalerel}
\usepackage{graphicx}
 \newcommand{\precneq}{%
\mathrel{\ooalign{$\preccurlyeq$\cr\kern1.2pt$\nback$}}}
 \usepackage{bigints}
\usepackage[scr=boondox,  
            cal=esstix]   
           {mathalpha}
           
            \newcommand{\rev}[1]{ {  {  #1}}}  
           
 \usepackage{fancyhdr}

     \if\drafts1
\fi
  \newcommand{\draft}[1]{\if\drafts1  {  #1}\fi}  
  \newcommand{\hide}[1]{}  
  \newcommand{\WSS}{weakly stationary }
  \newcommand{\Q}{  {\bf Q}}
 \newcommand{\U}{    {\bf U}}  
  \newcommand{\Mint}[1]{#1}  
  \newcommand{\AP}{uniqueness pair}
   \newcommand{\APs}{uniqueness pairs}  
  \newcommand{\m}{    \mathsf m}  
 \newcommand{\barr}[1]{#1}  
   
 \newcommand{\X}{   \mathsf X}  
   
\newcommand{\e}{\mathbf E}

\RequirePackage[ citecolor = blue,backref=page]{hyperref}

\newtheorem{question}{Question}
\newtheorem{corollary}{Corollary}
\newtheorem{definition}{Definition}
\newtheorem{example}{Example}

\newtheorem{lemma}{Lemma}
\newtheorem{proposition}{Proposition}
\newtheorem{remark}{Remark}
\newtheorem{theorem}{Theorem}
 \newcommand{\E}{  \mathscr  E}

 \newcommand{\z}{  \mathscr  z}  
   \renewcommand{\k}{   \mathsf k}  
   \newcommand{\F}{  \mathscr  F }  
\renewcommand{\P}{\mathsf{P}}
\newcommand{\M}{\mathsf{M}}

\newcommand{\Leb}{\mathcal{L}}

\newcommand{\C}{\mathcal C}

\newcommand{\x}{{\bf x}}
\newcommand{\y}{{\bf y}}

  \newcommand{\s}{  \mathscr  s}

\newcommand{\remove}[1]{}
 \newcommand{\sinc}{\textrm{sinc}}  
\renewcommand{\sp}{\textrm{Sp}}

\title{Maximal rigidity of random measures and uniqueness pairs: stealthy processes, quasicrystals and periodicity}

\author{Rapha\"el Lachi\`eze-Rey \thanks{Mathnet, Inria Paris, 49 Rue Barrault, 75013 Paris, France\\ \indent raphael.lachieze-rey@math.cnrs.fr}}

\date{}

\renewcommand{\C}{   \mathsf C}
\renewcommand{\S}{   \mathsf S}

\begin{document}
\maketitle

\label{here}

%

   \renewcommand{\c}{  \mathscr  c} 
 
   \label{ch:necessity}

 {\bf Abstract:} 
This article investigates the phenomenon of {\it maximal rigidity} in random spatial systems, where perfect interpolation of the process is possible from partial information, specifically, from its restriction to a strict subdomain, often resulting in a trivial tail $\sigma$-algebra. A classical example known since the 1930's is that a time series is  {\it deterministic}, i.e. fully determined by its values on the negative integers if its spectrum has a gap, or at least a sufficiently deep zero. We extend such results to higher dimensions and continuous settings by establishing a connection with the concept of {\it \APs}, rooted in the uncertainty principle of harmonic analysis. We present several other manifestations of this principle, unify and strengthen seemingly unrelated results across different models: stealthy processes are shown to be maximally rigid on cones,  {and for quasicrystals this cone can be arbitrarily small}; discrete integer-valued processes are necessarily periodic when they have a simply connected spectrum. Finally, we identify a surprising class of continuous fields with seemingly standard behavior, such as linear variance and finite dependency range, that undergo a phase transition: they are perfectly interpolable on $B(0,\rho)$ for $0<\rho \leqslant \frac{ 2}{\pi }$ but exhibit no rigidity for $\rho >2$.\\

 {\bf Keywords:} Rigidity, perfect interpolation, point processes, stealthy processes, quasicrystals, Gaussian fields

\section{Introduction}
\label{sec:intro}

The phenomenon of number rigidity for disordered point processes, or more generally that of $ k$-rigidity, has been intensively studied in the last years, see  \cite{GP17,GL-sufficient, Lr24} and references therein.
 Maximal rigidity, or its counterpart  {\it perfect interpolability}, is a more extreme phenomenon that concerns a specific class of models. Its first instance was proved independently by Kolmogorov  \cite{Kolmo41} and Wiener  \cite{Wiener}: a stationary sequence $\X =  \{\X_{ k};k\in \mathbb{Z} \}\subset \mathbb{R}$, or  {\it time series}, is entirely determined by its past, i.e. $\X\in  \sigma (\X_{ k},k<0)$,  if and only  if  the logarithmic integral of its spectral density diverges, which happens if it  has a gap or a sufficiently deep zero (see  \eqref{eq:dez-torus}). We  {  more generally say that a random measure or a random process $ \X$ on some subgroup of $ \mathbb{R}^{ d}$ is  {\it perfectly interpolable} from $ A$, or  {\it maximally rigid} on $ A^{ c}$, if the restriction of $ X$ to $ A$ allows to a.s. completely reconstruct all the process $ X$ (in the previous examples, we have perfect interpolability from the set of negative integers $A = \mathbb{Z} _{ -}$). We investigate here how some properties of $ \X$'s spectrum imply maximal rigidity behaviours.}

 We first give   a multi-dimensional version of this  result, replacing $ \mathbb{Z} _{ -}$ by $ (\mathbb{Z} _{  +  }^{ d})^{ c},$ also valid in the continuous space with $ (\mathbb{R}_{  + }^{ d})^{ c}$. It applies in particular to  {\it stealthy processes}, characterized as having a spectral gap.  In condensed matter physics, it reflects the transparency to wavelengths corresponding to the spectral gap, this topic attracted considerable interest, see the non-exhaustive bibliographic sample   \cite{TZS,ZST,ZST-I,Tor16b,Tor18,Morse,casiulis,Klatt-nature} . It has relations with the concept of  {\it blue noise} in image analysis and optimal transport  \cite{BlueNoise,deGoes}, quantization in numerical probability  \cite{Sarrazin}, or numerical integration, since stealthy samples yield a fast decay for linear statistics  \cite{VarianceMC,survey-hu}. Their study from the mathematical point of view is difficult and few rigourous results exist, see   \cite{GL18}, or  \cite{AdiGhoshLeb} in the discrete setting;
 there are  deep unresolved mathematical questions around them, such as the mere existence of stationary disordered models, see  \cite[Section 5.2]{survey-hu}.
 Ghosh and Lebowitz \cite{GL18} showed in particular that such processes are maximally rigid on bounded sets, i.e. the points on a given bounded set can be recovered from outside points. We extend this result by showing  that they are maximally rigid on strictly convex cones. 
 
 Another possible way to generalize  Kolmogorov-Wiener's theorem is when one assumes that a random measure has a spectral gap (or a sufficiently deep zero) along each coordinate. This assumption is much stronger, and it yields a stronger rigidity behaviour: we show that in this case we have perfect interpolability  \underline{from} a convex cone (e.g. $ \mathbb{Z} _{ -}^{ d}$). 
 It allows  in the discrete setting to generalise to higher dimensions a surprising result of Borichev et al. \cite{BNS} stating that an integer-valued stationary process whose spectral density has an exponentially deep zero, is in fact a.s. periodic, later refined in  \cite{BSW18}. With more general results from complex analysis, we show on $ \mathbb{Z}^{ d}$ that it is actually sufficient to assume that the spectrum is a subset of a simply connected  set of $  \mathbb  T  ^{ d}.$
 
 
 We then turn towards the study of  quasicrystals, characterised by a stronger constraint: the spectrum is supposed to be purely atomic. Quasicrystals are models of homogeneous media that emerged from experimental findings in materials science  \cite{schetchman}, after some predictions from crystallography  \cite{mckay}, see at table  \ref{table:1}-(c) the empirically measured spectrum of the AiMnPd alloy.
It drew considerable interest in both physics and mathematics, see for instance the survey  \cite{meyer-quasicrystals} and the monograph  \cite{baake-book}; they   turned out to be connected to  {an old line of mathematical investigations initiated with the 18th Hilbert problem, with works from Berger, Robinson, Wang, culminating with the aperiodic tilings discovered by Penrose}  \cite{penrose-tilings}, and generalise the  {\it cut-and-project} models introduced by Meyer  \cite{meyer72}. Mathematically, quasicrystals are uniformly discrete and relatively dense point configurations  having a purely atomic spectral measure.   Bj\"orklund  and Hartnick \cite{HartBjo}  studied stationarised    {\it cut-and-project} processes, and proved in particular that almost all such models are  hyperuniform, they also questioned whether they are number rigid, and it was answered in  \cite{Lr24} that they actually are maximally rigid on bounded sets, like stealthy processes. 
 {Hyperuniformity means that the number variance in large balls is negligible  with respect to the volume of the ball. This is a (generally) unexpected cancellation phenomenon, it indicates strong interactions between points, see for instance the surveys  \cite{Tor18,survey-hu}. It is related in the deterministic setting of almost periodic quasicrystals to the concepts of  {\it quick convergence to equilibrium}, or  {\it strict boundary property}, see for instance  \cite{Miekisz}, in particular definition (1).}
We show here that the rigidity properties of quasicrystals extend beyond bounded sets, or convex cones like stealthy processes, in that quasicrystals are rigid on concave cones: they can be entirely determined from their restriction on any cone with non-empty interior, such as $ \mathbb{R}_{ -}^{ d}$. This somehow means that all the information of a quasicrystal configuration is contained in any arbitrarily small sub-portion of it, and echoes a  line of results involved in determining what is the quantity of disorder in a quasicrystal with atomic spectrum \cite{Lagarias,KurSarnak,LevOlev}. All the findings described above imply for instance that the corresponding models have a trivial tail $ \sigma $-algebra.
 
 We finally exhibit a   surprising phenomenon for a  class of random fields which are rather standard (not stealthy, not even hyperuniform, with integrable or even finite-range covariance) and exhibit a phase transition: there is (maximal) rigidity on $ B(0,\rho )$  if and only  if  $ \rho \leqslant \rho _{c}$ for some $ \rho _{c}\in [2/\pi ,2]$, which means the field can be perfectly reconstructed on $ B(0,\rho )$. This phenomenon relies on Jensen's identity, yielding that for a non-zero entire function of type $\rho $  the density of the zero set     is bounded by $\pi \rho  /2$.\\
 
  {  The rest of Section  \ref{sec:intro} is devoted to giving the proper mathematical framework to study the phenomenon of maximal rigidity rigourously  in the general setting of weakly stationary random measures. For a first reading, one can only consider for $ \M$ a stationary point process. Once the concept of perfect interpolation is clear, the interested reader can skip directly to the applications of Sections  \ref{sec:stealthy} (Stealthy processes), \ref{sec:quasi} (Quasicrystals), \ref{sec:periodicity} (Periodic fields), or  \ref{sec:finite-range}  (Gaussian continuous fields).}
 As detailed at the end of the section, those results are related to establishing  {\it uniqueness pairs} in harmonic analysis, that is sets in the direct and Fourier space that cannot support respectively a non-null function and its Fourier transform.

 \subsection{Random fields and measures}  
Following Kolmogorov's isomorphism theorem in the 1930's,  
several authors in the 1950's and 1960's, such as Doob, Yaglom, or Gelfand, derived a spectral theory for generalised processes, with works from Bartlett in the 1960's targeted towards point processes. It relies on the fact that continuous positive definite bilinear forms can be represented by a non-negative measure, building on the work of Bochner in the 1930's, generalised by Schwartz to tempered distributions. See for instance Daley and Vere-Jones'book  \cite[Section 8]{DVj08} for a brief history about random measure.

We consider more generally a random    {\it weakly stationary}  complex-valued measure $ \M$, i.e. a collection of complex square integrable random variables $ \Mint{\M}(f)$ on a probability space $ (\Omega ,\mathbf{P}),$ for $ f$ in the space $ \mathcal{C}_{c}^{b}(\mathbb{R}^{d})$ of measurable bounded and compactly supported functions  satisfying   \begin{itemize}
\item[(i)] $\Mint{\M}$ a.s. induces a complex-valued measure on each compact of $ \mathbb{R}^{ d},$
\item [(ii)] $\textrm{Var}\left(\Mint{ \M}(\tau _{x}f) \right) =  \textrm{Var}\left(\Mint{\M}(f)\right)$ where $ \tau_{x}$ is the operator of translation by $ x\in \mathbb{R}^{d}$.
\end{itemize}

 This class encompasses locally square integrable (weakly) stationary marked point processes, spin measures,  continuous random fields or their nodal domains.
The class of admissible square integrable linear statistics $ \Mint{\M}(f)$ generally extends to a broader class of functions $ f$ with unbounded support, such as Schwartz functions.
By stationarity, a disintegration  yields that for each $ f,g\in \mathscr C_{c}^{ b}(\mathbb{R}^{ d})$,
\begin{align}
\label{eq:cov-measure}
 \textrm{Cov}\left(\Mint{\M}(f),\Mint{\M}(g)\right) =  \int_{ 
 }f(y)  \overline{  g(x + y)}\C(dx)dy,
\end{align}
where $ \C$ defines a signed measure on each compact subspace of $ \mathbb{R}^{ d}$ (but it might not be possible to give a sense to $ \C(\mathbb{R}^{ d})).$ We refer the reader to  \cite{DVj08} for non-negative random measures, or  \cite{AT07} for Gaussian processes, which is sufficient for the examples studied in this paper.
 It is symmetric, i.e. $ \C(A) = \C(-A)$ for $ A$ bounded, and called the  {\it (reduced) covariance measure}. We can see by taking for instance $ f,g$ that approximate Dirac masses in respectively $ 0$ and $ x$ that $ \C(dx)$ measures the covariance of the masses in $ 0$ and $ x$ (see also below an interpretation with discrete fields). 
 
Taking $ f = g$  in  \eqref{eq:cov-measure} yields that $ \C$ is semi-definite positive. The generalisation of Bochner's theorem  \cite[Th. 4.5]{Berg}  \footnote{not requiring positivity of $ \C$}   yields that there exists  a non-negative symmetric measure $ \S$ which satisfies a  {Parseval-type formula}, also called  {\it phase-space} formula
\begin{align}
\label{phase-continuous}
 \textrm{Var}\left(\Mint{\M}(f)\right) = (2\pi )^{-d}\int_{  \mathbb{R}^{ d} } | \hat f | ^{2}d\S,f\in \C_{ c}^{ b}(\mathbb{R}^{ d}),
\end{align} 
where $$ \hat f(u) = \int_{\mathbb{R}^{ d}}e^{-i   u\cdot x } f(x)dx,u\in \mathbb{R}^{ d} .$$ The measure $ \S$ is in fact the Fourier transform of $ \C$ in the sense of tempered distributions by the original Schwartz's result (\cite[Th. XVIII]{schwartz}), and  \eqref{phase-continuous} holds for all Schwartz functions also.

$ \S$ is   called the    {\it spectral measure} of $ \M$. It is translation bounded  (see for instance the proof at  \cite[Prop. 8.2.II-(iv)]{DVj08}), which implies the useful property
\begin{align}
\label{eq:integr-SF}
\int_{\mathbb{R}^{ d}} (1 + \|u\|)^{ -d-1}\S(du)<\infty .
\end{align}

%
%

%

On $ \mathbb{Z} ^{ d}$, a random field  is a collection of square integrable centred random variables $ \X = \{\X_{\m},\m\in \mathbb{Z} ^{d}\}$. Similarly, we say that it is   weakly stationary if for any function $ f:\mathbb{Z} ^{d}\to  \mathbb R  $ with finite support, the variance of the quantity 
\begin{align*}
\Mint{\X}(f): = \sum_{\m\in \mathbb{Z} ^{d}}\X_{\m}f(\m)
\end{align*}
  is unchanged when $ f$ is shifted by a quantity $ \m\in \mathbb{Z} ^{d}$, i.e. $  \textrm{Var}\left(\Mint{\X}(f)\right) =  \textrm{Var}\left(\Mint{\X}(\tau _{\m}f)\right)$. \\

Since some statements  are very similar in both frameworks, we introduce unified notation to treat both in parallel. Abstractly denote by $ \e^{ d}$ a space which can be either $  \mathbb R  ^{d}$ endowed with Lebesgue measure or $  \mathbb  Z  ^{d}$ endowed with the counting measure.  If $ \e ^{ d}= \mathbb{Z}^{ d} $, we  define for a weakly stationary random field $ \M$, $ \C(\m) =  \textrm{Cov}\left(\M_0,\M_\m \right),\m\in   \mathbb Z  ^{d}.$ Also denote, for $ B\subset \mathbb{Z} ^{ d}$, by $  \mathscr  C_{ c}^{ b}(B)$ the class of $ f:B\to  \mathbb C $ with bounded support.
 In this framework,  \eqref{eq:cov-measure} still holds with an integration over $ \e$.

Denote the dual group by $ \hat \e ^d= \mathbb{R}^{d}$ if $ \e ^d= \mathbb{R}^{d}$, and $ \hat  \e ^d=  \mathbb  T  ^{d}$ with $  \mathbb  T   =   { \{e^{ is},s\in [-\pi ,\pi )\}}$ if $ \e^{ d}  = \mathbb{Z}  ^{ d}.$ In the latter case, the Fourier transform is defined for $ f\in  \mathscr  C_{ c}^{ b}(\mathbb{Z} ^{ d})$  
\begin{align*}
 \hat f(u) :=\sum_{\k\in \mathbb{Z} ^{ d}}u^{ \k}f(\k),u \in  \mathbb  T  ^{ d},
\end{align*}
with the notation $ (u_{ 1},\dots ,u_{ d})^{ (k_{ 1},\dots ,k_{ d})} = u_{ 1}^{ k_{ 1}}\dots u_{ d}^{ k_{ d}}.$

Then  \eqref{phase-continuous} generalises to
\begin{align}
\label{eq:phase}
 \textrm{Var}\left(\M(f)\right)=(2\pi )^{-d}\int_{ \hat \e^{d}}| \hat f(  {  u})|^{2}\S(  { du})
\end{align}
for a bounded function  $ f$   that has a compact support, or, if $ \e ^{ d} =\mathbb{R}^{ d} $, that is a Schwartz function.
All the results of the present paper are actually about $ \S$,  they  imply rigidity behaviours in so much that  \eqref{eq:phase} is satisfied, i.e. we are dealing with  {\it linear rigidity}. If $ \e ^{ d}= \mathbb{Z} ^{ d}$, $ \S(  \mathbb  T  ^{ d}) =  \textrm{Var}\left(\M_0\right)<\infty $, which we also generalise with  \eqref{eq:integr-SF} to 
\begin{align}
\label{eq:integr-SF-general}
 \int_{ \hat \e^{ d}}\kappa _{  \hat \e^{ d}}(u)\S(du)<\infty 
\end{align}
with $ \kappa _{ \mathbb{R}^{ d}}(u) = (1 + \|u\|)^{ -d-1}$ and $ \kappa _{ \mathbb{T} ^{ d}}(u) = 1.$
By Radon-Nikodym's theorem, $ \S$ decomposes into a measure continuous with respect to Lebesgue measure $\s(u)du$, and a singular component $\S_{c}$; $\s$ is called the  {\it structure factor} when $ \M$ is a point process, and the  {\it spectral density} in greater generality. A  {\it spectral gap} is a nonempty open set $ O\subset \hat \e^{d}$ such that $ \S(O) = 0$.

  \subsection{Linear Maximal rigidity (LMR)}    
Given a \WSS random measure  $ \M$ and a subset $ B$  in $ \e^{ d}$, we say that $ \M$ is  {\it maximally rigid}  (MR) on $ B$  if the restriction $ \M1_{B}$ of $ \M$ onto $ B$ is completely determined  from $ B^{ c}$ :  $\sigma ( \M 1_{ B}) \subset   \sigma (\M1_{ B^{ c}})$.

We actually  deal in this paper with  {\it linear maximal rigidity} (LMR)   {on $ B$}, meaning that for any $ f \in \mathscr C_{c}^{ b}(B)$, there are $ h_{n}\in \mathscr C_{c}^{b }(B^{ c}),n\geqslant 1,$ such that  in $ L^{2}(\mathbf P ),$
\begin{align}
\label{eq:LMR-gamma}
\Mint{\M}(h_{n})  \xrightarrow[n\to \infty ]{} \Mint{\M}(f ). 
\end{align}
We shall also use the terminology that $ \M$ is  {\it perfectly interpolable from $ B^{ c}$}, as the focus can be on the minimal (class of) set(s) where it is sufficient to know $ \M$ to completely determine it on all the space.  In most cases studied here, we have rigidity at least on all compact $ B$ with non-empty interior, which implies that the tail $ \sigma $-algebra is trivial
: for any $ R>0$, $ \M$ is perfectly interpolable from $ B(0,R)^{ c}.$
In most cases, we are able to prove that even for some unbounded $ B$, typically a cone, we have rigidity on $ B$.


\subsection{Uniqueness pairs and plan of the paper}  
\label{sec:AP}

In harmonic analysis, an  {\it uncertainty principle} is a statement quantifying that a non-null function and its Fourier transform cannot be simultaneously too concentrated in some sense. 
 We are here interested in the following instance of this principle: Given two sets $ A, \hat A\subset \mathbb{R}^{d}$, or $ A\subset \mathbb{Z} ^{d}, \hat A\subset  \mathbb  T  ^{d},$ say they form a  {\it  uniqueness pair}  if there is no function $ f$ vanishing on  $ A$ with $ \hat f$ vanishing on $ \hat A$, except for $ f \equiv 0$. 
 
 \begin{remark}
The concept of uniqueness pair must be connected to a particular class of functions (Schwartz, $ L^{ 2}(\mathbb{R}^{ d})$, ....), but the classes of sets obtained are often similar (see for instance the proof of Theorem \ref{thm:stealthy-cone}). In this work, we use the concept of uniqueness pair only through Theorem  \ref{thm:general-rigidity} below, and we do not formalise further the class of functions involved. 
\end{remark}

 A well known example is when $ A^{ c}, \hat A^{ c}$ are compact on $ \mathbb{R}^{ d}$. More generally, for $ A\subset  \mathbb{R}^{ d} $  with $ A^{ c}$ compact, the Schwartz Paley-Wiener theorem implies that $ B(0,\varepsilon ),\varepsilon >0,$  forms a \AP~ with $ A$ for the class of tempered distributions.
Establishing \APs~ is essential for getting MR results due to the following characterisation of LMR, proved at the end of the section.

    \begin{theorem} 
   \label{thm:general-rigidity}
 
A \WSS random measure $ \M$ on $ \e^{ d}$ with spectral measure $ \S$ on $ \hat \e^{ d}$ is perfectly interpolable from $ A\subset \e^{d}$   if and only if  for every $ \varphi \in L^{2}(\S)  \setminus \{0\}$, $   {\rm Sp}(\varphi \S)\cap  A\neq \emptyset   $  
, where the Fourier transform and spectrum  {  \text{\rm{Sp}}} are understood in the sense of tempered distributions.        
  \end{theorem}
Therefore we have a recipe for establishing maximal rigidity: for $ \hat A$ forming a \AP~ with $ A,$ if $  \S$ (and therefore $ \varphi \S$) vanishes on $ \hat A$, either $ \varphi \S \equiv 0$ or $ \sp(\varphi \S) \cap A\neq \emptyset $, which means  perfect interpolability from $ A$ by the theorem above.

Let us give an example:  Runge-Weierstrass's theorem yields that for $ \hat A\subset  \mathbb  T  $ open non-empty, a continuous function can be uniformly approximated by polynomials on $  \mathbb  T   \setminus  \hat A$.  Kolmogorov's isomorphism theorem then yields that a time series $ \X$ whose spectrum vanishes on $ \hat A$ is perfectly interpolable from the discrete half line $ A = \mathbb{Z} _{ -} $ of negative integers (or any other half-line).
Szeg\"o  \cite{Sz20} extended the approximation  result  for absolutely continuous measures by proving that the gap assumption can be relaxed to the divergence of the logarithmic integral of the density
\begin{align}
\label{eq:dez-torus}
  \int_{  \mathbb  T  }\ln(u)\s(u)du = -\infty .
\end{align}
Verblunsky  \cite{Verblunsky} generalised it to an arbitrary finite measure $ \S$ which continuous component satisfies   \eqref{eq:dez-torus}.
Through the linear isomorphism, Kolmogorov  \cite{Kolmo41} and Wiener   \cite{Wiener} recovered this result, yielding the following theorem. See a bibliography in  \cite[Appendix, XII]{Doob}.  {  Recall that any non-negative measure $ \S$ can be decomposed in a  components having a density $ \s$  with respect to Lebesgue measure and a component singular with $ \Leb$.}

\begin{theorem}[Kolmogorov-Wiener]
\label{thm:Szeg\"o-classic}
Let $ \X_{\m},\m\in \mathbb{Z} $ a \WSS random field. Then $ \X_{\m}$ is LMR on $ \mathbb{Z} _{-}$  if and only if   the continuous component $ \s\Leb$ of its spectral measure satisfies  \eqref{eq:dez-torus}. 
 \end{theorem}  
   
This condition is satisfied if for instance  $ \s(u) \sim e^{	-1/ { | u | }}$ as $ u\to 0$, but  $\s(u) \sim e^{-1/ \sqrt{| u |} }$ is not enough.  If one only assumes that all  the derivatives vanish, one  ``only'' has rigidity on a bounded set   (see Corollary \ref{cor:determining} for a more general principle). Wiener  \cite{Wiener} also gives a version of the result for   continuous random processes on $ \mathbb{R}$. Krein  \cite{Krein} shows that it extends to more general processes, see Section  \ref{sec:SW} ( Table \ref{table:1}-(a)).\\

 Let us now put  the results in perspective with this concept and give a plan of the paper. 
 Results and related uniqueness pairs are graphically illustrated  in Table   \ref{table:1}. \begin{itemize}
\item We consider the generalisation to random measures on $ \mathbb{R}^{ d}$ or $ \mathbb{Z} ^{ d}$ having a spectral gap, called  {\it stealthy random measures}. A result of Shapiro \cite{Shapiro},  that can be extended to tempered distributions,  says that   for a strictly convex cone $ B\subset \mathbb{R}^{d}$, $ A: = B^{ c}$ forms a \AP ~with any ball $ \hat A  = B(u,\varepsilon ),\varepsilon >0$. This implies that stealthy processes are maximally rigid on strictly convex cones, extending previous results on a bounded $ B$ (Theorem \ref{thm:stealthy-cone}), see Table  \ref{table:1}-b. In the discrete setting, Helson and Lowdenslager \cite{Helson} show that convex cones can actually be enlarged to be ``almost'' half-spaces. 

\item One can formulate a different multi-dimensional version of Szeg\"o's theorem adapted to quasicrystals, or more generally to any random measure whose spectrum is purely atomic.  While the previous ``stealthy'' result showcases processes that are perfectly interpolable on a convex cone, in Section  \ref{sec:stealthy} we show that quasicrystals are perfectly interpolable  \underline {from} an arbitrarily small cone with nonempty interior, which is clearly stronger.    This ``quasicrystal'' result is directly  proved with some complex analysis exploiting the atomic hypothesis.  See Table  \ref{table:1}-c.

\item We then investigate an intriguing phenomenon for discrete fields with values in $ \mathbb{Z} $. It was proved by  \cite{BSW18,BNS}  that if a stationary field $ \mathbb{Z} \to \mathbb{Z} $ has a spectral gap, it is a.s. periodic, and the spectrum is actually discrete. The generalisation to $ \mathbb{Z} ^{ d}$ is more subtle. If the spectrum  is only assumed to contain a gap, or even a half space, it is not enough to guarantee periodicity. If on the other hand the spectrum is a simply connected subset of $  \mathbb  T  ^{ d}$, then periodicity is enforced (Theorem \ref{thm:BSW-d}).  
A typical example is when the spectrum has a cut along each coordinate, as in Table  \ref{table:1}-d.
The proof relies on the fact that such processes exhibit a strong form of perfect  interpolability  \underline {from} the orthant $ \mathbb{Z} _{ -}^{ d}$, like quasicrystals.   This is proved by considerations of approximability by complex polynomials of several variables. A result of Stolzenberg \cite{Stolzenberg}   on polynomially convex sets allows to show that simple connectedness of the spectrum is enough to ensure periodicity. We also give an example of a spectrum which is not a subset of a simply connected set and yet enforces periodicity, see Table  \ref{table:1}-e.
 
\item Finally, in Section \ref{sec:finite-range}, 
we showcase a  completely different and striking form of rigidity experiencing a phase transition:
 if you take  a random field on $ \mathbb{R}^{d}$ whose covariance is, say, the truncated cone function $\C(x) = (1-\|x\| )\mathbf{1}_{\{ \|x\|\leqslant 1\}}$, then it is MR on the ball $ B(0,\rho )$ if $ \rho <2/\pi $, and it experiences no form of   rigidity for $ \rho >2$ (Theorem  \ref{thm:short-range}). The rigidity result relies on the fact that  $ A = B(0,2)$   forms a \AP~ with  the union of concentric  spheres $ \hat A = \cup _{ k\in \mathbb{N}}\partial B(0,k\pi /2)$ (see Table \ref{table:1}-(f)), or said differently if a tempered measure has its spectrum in $ B(0,2 )$, it cannot vanish on all circles of radii $  \frac{\pi }{2} k$. These results are a classical consequence of the Schwartz-Paley-Wiener theorem and Jensen's identity in complex analysis. A similar result can be stated for a well chosen $ \mathcal{C}^{ k}$ field for any $ k\in \mathbb{N}.$
\end{itemize}
The question of determining \APs~$ A, \hat A$ has generated a lot of activity in the 70's and 80's, mostly in dimension $ 1$, see the monograph  \cite{Havin} where $ (A^{ c}, \hat A^{ c})$ is called a  {\it mutually annihilating pair}. There is also the celebrated result of Benedicks that states that if $ A, \hat A\subset \mathbb{R}^{ d}$ have  finite Lebesgue measure, they form a \AP  ~\cite{Benedicks}, generalised in higher dimensions in  \cite{Jamming}.  
 In terms of random measures, a version adapted to distributions could mean with Theorem \ref{thm:general-rigidity} that if a random measure's spectrum has finite volume, then it is MR on every set with finite volume, most likely under some additional assumptions. See also  \cite{Helson} for a result on discrete half-spaces.
 More recently, and for other purposes,  the celebrated results of  Radchenko and Viazovska exploit new findings on uniqueness pairs, namely they show
 that $ A =  \hat A = \{0, \pm \sqrt{1}, \pm \sqrt{2}, \pm\sqrt{3}, ...\}$ form a uniqueness pair  \cite{Viazovska}, followed by   other works, see for instance  \cite{Sodin-uniqueness} and references therein.  
 
\subsection{Proof of Theorem \ref{thm:general-rigidity} and monotonicity of rigidity}
 
 The proof relies on the fact that LMR on $ A$  is equivalent to having \eqref{eq:LMR-gamma} for all $ f \in \mathscr C_{c}^{ b}(A )$. By  \eqref{eq:phase}, this is equivalent to having $ \hat f $  in the $ L^{ 2}(\S)$-closure of $H_{ A}$ where \begin{align*}H_{A} = \{ \hat h:h\in \mathscr C_{c}^{b}(A^{c})\},
\end{align*} 
for all $ f\in \mathscr C_{c}^{ b}(A) $, and therefore this is equivalent to having  $ H_{ A}^{ \perp} = \{0\}$ in $ L^{ 2}(\S)$. We therefore proved that LMR is equivalent to: for all $ \varphi \in L^{ 2}(\S)$ such that 
\begin{align}
\label{eq:equiv-orthog}
 \int_{  \e  ^{ d}} \hat h \varphi \S = \F(\varphi \S)( h) = 0,\text{\rm{ for all }} h\in \mathscr C_{c}^{ b}(A^{ c})
\end{align}
we have $ \varphi \S \equiv 0$.
Clearly  \eqref{eq:equiv-orthog} is equivalent to $ \F(\varphi \S)$ being supported by $ A$, which concludes the proof.\\

Since this is only a property of $ \S$, we shall equivalently  define the assumption of Theorem \ref{thm:general-rigidity} as the LMR induced by the non-negative measure $ \S$ on $ \e$.
We derive the following corollary:
\begin{corollary}
\label{cor:monotone}
Let $\S,\S'$ non-negative symmetric measures on $ \hat \e^{ d}$ such that $\S'$ has a bounded density $f$  with respect to $\S$, and let $ A\subset \e^{d}$. Then if $\S$ induces LMR on $ A$, so does $\S'$. 
\end{corollary}

\begin{proof}
Assume $\S$ is LMR and let $\varphi \in L^{2}(\S')$  {with $ \sp(\varphi \S')\subset A$}.
 Since $f\leqslant c<\infty $, $\varphi f\in L^{2}(\S)$:
\begin{align*}
\int_{}(\varphi f)^{2}d\S\leqslant c \int_{}\varphi ^{2}fd\S =c \|\varphi \|_{L^{2}(\S')}^{2}<\infty ,
\end{align*}
and $ \sp(\varphi f\S)\subset A$.
%
Hence $\varphi f \equiv 0$ in $L^{2}(\S)$, hence $\varphi \equiv 0$ in $L^{2}(\S')$.
\end{proof}

  \subsection{Reconstruction procedure}
  \label{sec:reconstruction}
  
  Linear maximal rigidity on $ A$ implies that any linear statistic supported by $ A$ can be approximated by linear statistics supported by $ A^{ c}$. Finding them is a matter of Hilbert analysis. Say for instance $ A = B(0,1)$ and one wishes to know the mass $ \M(B (0,1/2 ))$ based on the information contained in $ B(0,1)^{ c}$. Then one must project $ { 1_{ B(0,1/2)}}$ onto a Hilbert basis $ (h_{ k})_{ k\geqslant 1}$ supported by $ B(0,1)^{ c}$, where the scalar product is $$ \langle h,h' \rangle _{ \S}:=  \textrm{Cov}\left(\M(h),\M(h')\right) = (2\pi )^{ -d}\int_{ } \hat h \overline{\hat h'} d\S.$$ Then the reconstruction is in $ L^{ 2}(\mathbf P )$
\begin{align*}
 \M({ B(0,1/2 )}) = \sum_{k} \langle 1_{ B(0,1/2 )}, h_{ k} \rangle_{ \S} \M(h_{ k}).
\end{align*}
When one wishes to estimates the moments $ \M(x  \mapsto  x^{ \k}1_{ B(0,1)}(x)),\k\in \mathbb{N}^{ d}$, the estimator only depends on the behaviour of $ \S$ around $ 0$  \cite{GP17,GL-sufficient,survey-hu}, so it is sufficient to build the Hilbert basis for a more simple measure $ \S'$ that is equivalent to $ \S$ around $ 0.$

     \section{Multi-dimensional interpolation theorems on cones}
     \label{sec:cones}

   Krein generalised the Kolmogorov-Wiener Theorem \ref{thm:Szeg\"o-classic}   in the continuous space. Let $ \e_{ -} = \e\cap (-\infty ,0),\e_{  + } = \e\cap (0,\infty ).$  \begin{theorem} [Krein]
  \label{thm:SW}
   Let $\M$ a \WSS random measure on $\e$. Then $ \M$ is LMR on $\e  _{  + }$  if and only if  $\M$'s spectral density $ \s$ satisfies
\begin{align}\label{eq:wiener}\int_{  \hat \e  }{    \ln(\s_{}(u))   }\kappa_{ \e} (u)du =- \infty .
\end{align} 
   \end{theorem} 
   
   The discrete case is treated in  \cite{Kolmo41}, and the continuous case is a consequence of  \cite[Theorem 1]{Krein}.  See \cite[XII-Thm 5.1]{Doob} for a proof in English for weakly stationary processes.  The case of a   general stationary measure $ \M$ can be deduced from the weakly continuous case by convolving $ \M$ with the kernel $ 1_{ [0,\varepsilon ]}(\cdot )$ and letting $ \varepsilon \to 0$ (but this is still a consequence of  \cite{Krein}).\\
    
Even if  it does not formally imply that $ \s$ vanishes at exponential speed, we sometimes say that $ \s$ has an  {\it exponentially deep zero} (EDZ) if  \eqref{eq:wiener} is satisfied.
We give here generalisations  in several dimensions, and present maximal rigidity results involving  convex cones. 

\subsection{Tensorisation}  
 \label{sec:SW}
We just saw that if the spectral density $ \s$ of a 1D random measure $ \M$   has a gap, or more weakly if its logarithmic integral diverges, then $ \M$ is perfectly interpolable from  the (discrete) half line. We shall generalise this result on $ \e^{d}$: 

\begin{proposition}
\label{prop:suff-DEZ-all}  $ \M$ is perfectly interpolable from the orthant $ \e_{ -}^{ d}$ if  
\begin{align}
\label{eq:tensor-bound}
 \S\leqslant \S_{ 1}\times \dots \times \S_{ d}
\end{align} where each $ \S_{ i}$ is a symmetric non-negative    measure on $ \hat \e$  satisfying  \eqref{eq:integr-SF-general} and \eqref{eq:wiener}. 
\end{proposition}  
  This is related to the fact that $ (\e^d_-)^{ c}$ forms a uniqueness pair with any  $ B\subset \hat \e^{ d}$ having a gap in each coordinate.
This proposition  is proved at Section \ref{sec:prf-prop-1}  .  
%
 
\subsubsection{Counter-example}
  
In dimension $ 1$, condition  \eqref{eq:wiener} only involves the spectral density. We show below that a gap in the spectral density is not sufficient in dimension $ d\geqslant 2$.  
 In this example,
 there is no EDZ along one dimension, and we do not have MR from a cone, not even from a half space.
\begin{proposition} 
\label{prop:separable}
Let $ \s_{ 1}:  \hat \e  \to  \mathbb R  _{  + }$ not satisfying  \eqref{eq:wiener}.
If for some non-null measure $ \tilde \s$ on $   \hat \e^{ d-1}$, 
$
\S\geqslant  \s_{ 1}  \times   \tilde \s,$ then $ \S$ is not perfectly interpolable from  the half space $ \e  _{ -}\times  \e^{ d-1} $ (and a fortiori not from  $  \e_{ -}^{ d}$ either).  
\end{proposition}


\begin{remark}
We give at Section \ref{sec:counter-ex-SC} the example of a spectrum with maximal rigidity from $ \mathbb{Z} _{ -}^{ d}$ which has a ``crossing'' in the $ u_{ 2}$ coordinate, in the sense that the spectrum is not simply connected on $  \mathbb  T  ^{ 2}$ as above, but the ``crossing path'' varies with the $ u_{ 1}$-level (Figure  \ref{fig:counter-ex}). Hence it seems that what really matters is whether the crossing path is straight or not.
\end{remark}

\begin{proof}
Theorems \ref{thm:SW} and \ref{thm:general-rigidity} yield that there is $ \varphi _{ 1}\in L^{ 2}(\s_{ 1})  \setminus \{0\}$ with $  { \rm supp} (\widehat{ \varphi_{ 1} \s_{ 1}})\cap  \e _{  - } = \emptyset $.  
 {Let $ \varphi _{ 2}\in  L^{ 2}( \tilde s)$ non-null, and $ \varphi  = \varphi _{ 1}  \otimes \varphi _{ 2}$.}
  For $ x=(x_{1}, \tilde x)\in \e\times   \e^{d-1} $
\begin{align*}
\F(\varphi \S) (x)=  \widehat{ \varphi _{1}\s_{ 1}  }\otimes  \widehat{\varphi _{ 2} \tilde \s  } (x) =  \widehat{\varphi _{1}\s_{ 1}}(x_{ 1}) \widehat{\Mint{\varphi _{ 2}} \tilde \s }( \tilde x) = 0\text{\rm{ if }}x_{ 1}<0 
\end{align*} by construction of $ \varphi _{ 1}$. We hence have found $ \varphi \in L^{ 2}(\S)$ such that $ \F(\varphi \S)$ is supported by $ (\e_{ -}\times \e ^{ d-1})^{ c}$, which yields that $ \S$ is not perfectly interpolable from the half space $ \e_{ - }\times \e^{ d-1} $ (Theorem \ref{thm:general-rigidity}).
\end{proof}
 
 \subsubsection{Proof of Proposition  \ref{prop:suff-DEZ-all}} 
 \label{sec:prf-prop-1}  
 
 Let $ \varphi\in L^{ 2}(\S) $ whose spectrum is supported by $ (\e_{ -}^{ d})^{ c}$. 
According to Theorem \ref{thm:general-rigidity}, we must prove $ \varphi \equiv 0$. We reason by contradiction. Then there is $ a<0$ and  $ \gamma _{ 1},\dots ,\gamma _{ d}$ functions with   support in $ [a,\infty )$(and Schwartz if $ \e = \mathbb{R}$) such that, with $ \Gamma (x) := \gamma _{ 1}(x)\dots \gamma _{ d}(x)$,
\begin{align*}
 \varepsilon : =  \left|
\int_{\e^{ d}}\Gamma \F(\varphi \S)
\right| >0.
\end{align*}
 Assume without loss of generality $ \varepsilon \leqslant  \frac{ 1}{ 2}   $ and  $ \|\gamma _{ i}\|\leqslant   \frac{ 1}{ 2}   .$
By  Theorem \ref{thm:SW}, for $ \varepsilon '>0,$
each $ \hat \gamma_{ i} $ can be approximated in $ L^{ 2}(\S_{ i})$ by  $ \hat   h_{ i}$ for some $  h_{ i}\in \mathscr C_{c}^{ b}((-\infty ,a])$ :
\begin{align*}
 \int_{ \hat \e} | \hat h_{ i}- \hat \gamma_{ i}  | ^{ 2}d\S_{ i}<\varepsilon '.
\end{align*}
Define the tensor product $ H  =  h_{ 1}  \otimes \dots  \otimes  h_{ d}$ so that $ \hat H = \hat h_{ 1}  \otimes \dots  \otimes \hat h_{ d}.$
Since $ \F(\varphi\S) $ is supported by $ (\e_{ -}^{ d})^{ c}$ and $ H$ by $ (-\infty ,a]^{ d}$, we have $  H \F(\varphi\S)  \equiv 0$, hence $ \int_{} \hat H \varphi d\S = 0$ and Plancherel's identity yields
\begin{align*}
 (2\pi )^{ d}\varepsilon = \left|
  \int_{} \hat \Gamma   \varphi d\S  
\right|\leqslant  \left|
\int_{}   \varphi  \hat H    d\S
\right|+\int_{} | \varphi  |   |  \hat H- \hat \Gamma  | d\S
 \leqslant  0 + \|\varphi \|_{ L^{ 2}(\S)}\sqrt{\int_{} | \hat H- \hat \Gamma  | ^{ 2}d\S}.
 \end{align*} 
We use the bound, with $  \delta _{ i}: = \hat \gamma _{ i}- \hat h_{ i},$
\begin{align*}
  | \hat \Gamma - \hat H |  \leqslant \sum_{i = 1}^{ d-1} | \hat \gamma _{ 1}\dots \hat \gamma _{ i-1} \delta _{ i} \hat h_{ i + 1} \dots  \hat h_{ d} | 
\end{align*}
and $  | \hat h_{ i} | \leqslant  | \delta _{ i} |  +  | \hat \gamma _{ i} | $. The assumption  \eqref{eq:tensor-bound} gives, for some $ c<\infty ,$
\begin{align*}
  \int_{} |  \hat \Gamma - \hat H | ^{ 2}d\S<cd  \max_{ i}\int_{} |  \delta _{ i} | ^{ 2}d\S_{ i} 
  \max_{ j}\max\left(
\int_{ \hat \e}| \hat h_{ j} | ^{ 2}d\S_{ j},\int_{ \hat \e}  | \delta _{ j} | ^{ 2}d\S_{ j}
\right)^{ d-1}<cd\varepsilon' (1/2)^{ d-1} \end{align*}
which gives a contradiction by choosing $ \varepsilon '  = \varepsilon/(cd2^{ d-1}) $.

  \subsection{Stealthy random measures  }
  \label{sec:stealthy}

 Stealthy systems are characterised in condensed matter physics as being transparent to a certain band of wavelengths,  {   which translates in the Fourier domain by having a spectral gap}.
It is reflected by the following general definition:

  \begin{definition}
  A random stationary measure $ \M$ on $\mathbb{R}^{d}$ is   {\it stealthy} if   {  $\text{\rm{Sp}}(\M)\neq \mathbb{R}^{ d}$}.
  \end{definition}  
Often, the gap is assumed to contain $0$, but it bears no importance on the rigidity results.  A non-exhaustive bibliographic sample from the physics literature about stealthy processes is  \cite{TZS,ZST,ZST-I,Morse,Klatt-nature}, see also more background in the introduction.

Mathematically, a typical example is the following: a shifted lattice is a point process of the form $\Lambda +U$, where $\Lambda \subset \mathbb{R}^{d}$ is a lattice and $U$ is uniformly distributed in the fundamental cell of $ \Lambda $  {  to ensure formal stationarity}. 
Such a model, or a finite union of independent shifted lattices, can be proven to be stealthy readily as its spectral measure is $ \Lambda ^{ *}  \setminus \{0\}$, where $ \Lambda ^{ *}$ is the dual lattice, by an application of Poisson's summation formula.  Physicists  put in evidence stealthy models of point processes which are also disordered, e.g. isotropic and mixing, through some simulation procedures.   On the other hand, it is difficult to rigourously establish the existence of disordered stealthy point processes.

  The interest also emerged in the mathematics literature,
 \cite{GL18} proved that stealthy point processes are maximally rigid on a compact $ A$,  and their result likely extends to more general random measures. In   \cite{AdiGhoshLeb}, the authors rather investigate those processes in terms of entropy per site, but still ask {\it how rigid are stealthy hyperuniform processes?}  
We prove here that stealthy processes (whether point processes or more general) are MR on unbounded domains: rigidity actually occurs on any closed strictly convex cone, i.e. a closed convex cone $ B$ not containing both $ x$ and $ -x$ for some $ x\neq 0.$ 
 In other words, for any closed strictly convex cone $ B$, $ A = B^{ c}$ forms a \AP~ with any non-empty open $ \hat  A$, see Table \ref{table:1}-(b).

   \begin{theorem}
   \label{thm:stealthy-cone}
    Let $ \M$ a stealthy \WSS random measure. Then for all closed strictly convex cone $ B$, $\M$ is LMR on $B$.   
    \end{theorem}

  The proof relies on a theorem of Levinson, generalised to higher dimension in  \cite{Shapiro}, that we further generalise to tempered distributions. It can be seen as a way to extend the previous theorems to higher dimensions. 
    Shapiro in fact states his result in terms of  {\it minor cone}, defined as a cone $ B$ such that for some $ t_{ 0}\in \mathbb{R}^{ d}$
\begin{align}
\label{eq:minor-cone}
 \inf_{ t\in B} \langle t_{ 0},t \rangle>0.
\end{align}

        \begin{theorem}
        \label{thm:shapiro}[\cite{Shapiro},Theorem A']
     Let $ B$ a minor cone, and $ \psi :\mathbb{R}^{d}\to \mathbb{C}$ a tempered function such that for some $ \delta >0$,
\begin{align*}
\int_{ B^{ c}} | \psi (t) | e^{\delta  | t | }dt<\infty 
\end{align*}
and such that $ \F \psi  = 0$ on some nonempty open set. Then $ \psi \equiv 0.$
     \end{theorem}  
    
    The  strict convexity assumption is essential as it is  easy to build Schwartz functions $\psi $ such that $\psi $ and $ \hat \psi $ both vanish on a half space, see Proposition  \ref{prop:separable}.
  The cone $ B$ in  Shapiro's theorem can be taken closed without loss of generality.  Let us show that a closed minor cone is the same as a  closed strictly convex cone. A closed cone which is not stricly convex contains $ x$ and $ -x$ for some $ x\neq 0$, which directly contradicts  \eqref{eq:minor-cone}. Conversely, for a closed strictly convex cone $ B$, the dual cone $ B^{*} = \{t: \langle t,x  \rangle\geqslant 0,x\in C\}$ has a non-empty interior. Let $ t_{0}$ be in this interior. If there is $ t\in B  \setminus \{0\}$ with $ \langle t_{0},t \rangle = 0$, let $ u = t_{0}-t$, and $ \varepsilon >0$ such that $ t_{0} \pm \varepsilon u\in B^{*}$. We have 
\begin{align*}
\langle t_{0} + \varepsilon u, t \rangle = \varepsilon \langle u,t \rangle = -\varepsilon \langle t,t  \rangle<0,
\end{align*}
 which contradicts the definition of $ B^{*}.$ Hence $ B$ is minor.

    \begin{proof}[Proof of Theorem  \ref{thm:stealthy-cone}]
 Let   $ \varphi  \in L^{ 2}(\S),\Psi = \varphi \S.$  Assume $ \F(\Psi)$ is supported by $ B.$ To conclude the proof it suffices by Theorem \ref{thm:general-rigidity} to prove $ \Psi\equiv 0.$
Knowing that $ \Psi$ has a gap $ B(u_{0},\varepsilon )$ in its support by the stealthiness assumption, if $ \Psi$ had a density $ \psi $, Shapiro's theorem would allow to conclude. Let us show that Shapiro's  theorem applies to general distributions.
 In the general case, let $ \kappa $ a Schwartz function supported by $ B({0},\varepsilon /2)$. Then by   \cite[Th. 7.19-(c)]{Rudin2}, $$ \psi (u):= \Psi  \ast  \kappa (u)= \Psi ( \tau _{-u}  \kappa )$$ defines a tempered function with Fourier transform $ \F\psi = \hat \kappa \F\Psi$.  It has a spectral gap in $ B(u_{0},\varepsilon /2)$ because $ \tau _{-u}\kappa $ has support in $ B(u_{0},\varepsilon )$ for $ \|u-u_{0}\|\leqslant \varepsilon /2$. Also
$ 
\F \psi  = \hat \kappa \F\Psi
$
vanishes on $ B^{c}$, as $ \F\Psi .$ Hence $ \psi \equiv0$ for all such $ \kappa $, meaning $ \hat \kappa \F\Psi\equiv 0$ for all such $\kappa $, it implies $ \F\Psi = 0$, hence $ \Psi = 0,$  which concludes the proof.     \end{proof}

    \subsubsection{  {  Maximal rigidity on a compact}}
 
 The gap assumption for a stealthy measure is physically relevant due to the special optical properties it conveys. Still the maximal rigidity behaviour on a compact remains if this assumption is relaxed in some ways that we explore here. First of all, one can make an assumption only on the continuous part $ \s$ of $ \S$. Second, one can relax the property that the zero is a spectral gap to being a deep zero, i.e. where all the derivatives vanish, or a set of zeros which is large, without necessarily containing a nonempty open set. 
  Say more generally that $ \s$ is  {\it determining}  if $ L^{2}(\s^{-1}) $ contains no other entire function than $ 0$ (under the convention $ 0^{ -1} = \infty $).
  
   \begin{corollary}
 \label{cor:determining}Let $ \M$ a \WSS random measure with spectral density $ \s.$
If $ \s$ is determining and $ A$ bounded, then $ \M$ is MR on bounded $ A$. 
  \end{corollary}

   This is the case if for instance  $ \s_{}$ has a  {\it  deep zero} ,  i.e. if for some $ u_{0}\in \mathbb{R}^{d},k\in \mathbb{Z} $
\begin{align*}
\int_{ B(u_{0},\varepsilon )} \frac{ \|u-u_{0}\|^{k}}{\s(u)}du = \infty.
\end{align*}
 {  This assumption can be seen as a weak form of stealthiness.}
It recovers the results of maximal rigidity of  \cite{GL18}.  
   
   \begin{proof}
Let $ \varphi \in L^{ 2}(\S)$.      Schwartz'Paley-Wiener theorem implies that if $  {  \sp}( \varphi \S)$ has a bounded support, then $\psi : =  \varphi \S$ is an entire function, hence $ \psi  = \varphi \s$, and 
\begin{align*}
 \int_{} | \psi |  ^{ 2}\s^{ -1} = \int_{} | \varphi |  ^{ 2}\s\leqslant \|\varphi \|^{ 2}L_{ 2}(\S)<\infty .
\end{align*}
By the  assumption, $\psi  =  \varphi \S \equiv 0$, hence by Theorem \ref{thm:general-rigidity}, $ \M$ is LMR on $ A.$    
   \end{proof}

 Call  {\it uniqueness domain} of $ \mathbb{R}^{d}$ a set $ D\subset \mathbb{R}^{d}$ such that if a real entire function $ \psi $ vanishes on $ D$ then $ \psi \equiv 0.$ 
 We see that if $ \s$ vanishes on a uniqueness domain, it is by definition determining.  In any case, a uniqueness domain forms a uniqueness pair with any set having a gap by Schwartz' Paley-Wiener  theorem.
 In dimension $ 1$, using the Weierstrass factorisation theorem, uniqueness domains are   sets with concentration points; in dimension $ d\geqslant 2$, the situation is more complicated, there are in particular no entire function with zeros isolated in $  \mathbb C ^{d}$  \cite{Lelong}, but any set with positive Lebesgue measure on $ \mathbb{R}^{d}$ is a uniqueness domain by 
 \cite{Mityagin}. 

 \section{Quasicrystal interpolation from minor cones}
 \label{sec:quasi}

Quasicrystals are broadly speaking atomic measures whose spectrum is purely atomic, supposed to reflect some  {\it aperiodic order}. Their study emerged after experimental discoveries in physics in the 80s and is related to many fields, including crystallography, aperiodic tilings, almost periodicity, see  Section \ref{sec:intro} for more background. Such objects are traditionally assumed to be homogeneous in space, and it is thus natural to consider random constructions that are invariant under translations  (\cite{HartBjo,torquato-quasi}). 
  Maximal rigidity on a compact $ A$ was shown in  \cite{Lr24}, the compactness of $ A$ allows to use the Paley-Wiener theorem for entire functions of exponential type. We rather investigate here maximal rigidity on unbounded $ A$, where Paley-Wiener theorem does not apply and complex analysis cannot be used directly anymore.

\begin{theorem}
\label{thm:quasi}
 Let $ \M$ a \rev{\WSS} random measure whose spectral measure $ \S$ is purely atomic. Then $ \M$ is perfectly interpolable from any cone with non-empty interior (Table \ref{table:1}-(c)).
\end{theorem}

   {   Before giving the mathematical proof, let us discuss the relation with the traditional deterministic setting of quasicrystals. Let $ P$ a  tempered measure on $ \mathbb{R}^{ d}$  having a discrete pure point Fourier transform $ \hat P = \sum_{i}a_{ i}\delta _{ s_{ i}}$ in the  sense of distributions (with the convention that $ s_{ 0} = 0$).  Then we consider for $ n\in \mathbb{N}$ the randomly translated version $   \tau _{ n U}P $ where $ U$ is random and uniform in the unit ball. 
We make the classical assumption that $ P$ is  {\it translation bounded}, i.e. 
\begin{align}
\label{eq:QC-density}
 \forall  r>0,\sup_{ x\in \mathbb{R}^{ d}}P(B(x,r))<\infty .
\end{align}
 By \cite[Lemma 14.15]{kallenberg2002foundations}, the random processes $ \tau _{nU }P$ admits at least a limit point $ \P$ in distribution in the vague topology, i.e. there is a subsequence  converging in law to $ \P$. There is no unicity in general, but there is if one additionally requires ergodicity.

 \begin{proposition}
  Let $ P$ a translation bounded measure with discrete spectrum.
   Then   any limit point $ \P$ of $ (\tau _{ nU}\P)_{ n\geqslant 1}$ is  stationary and has atomic spectrum, hence 
 the result of Theorem  \ref{thm:quasi} applies: $ \P$ is perfectly interpolable from any cone $ C$ with non-empty interior.   \end{proposition}

\begin{proof}
\emph{Step 1: existence of limit points.}
Since $P$ is translation bounded, for every compact $K\subset\mathbb R^d$,
\[
\sup_{x\in\mathbb R^d} (\tau_xP)(K)<\infty.
\]
By the standard compactness criterion for Radon measures in the vague topology  \cite[Lemma 14.15]{kallenberg2002foundations}, the translation orbit
\[
\{\tau_xP:x\in\mathbb R^d\}
\]
has relatively compact closure. Hence the laws of $\P_n=\tau_{nU}P$ are tight, so at least one subsequence converges in distribution to some $ \P.$

\emph{Step 2: stationarity of the limit.}
Let $F$ be a bounded continuous functional on the space of Radon measures equipped with the vague topology. Write $B_n:=B(0,n)$.
Since $nU$ is uniform on $B_n$,
\[
\mathbf E\left[F(\tau_x\P_n)\right]
=
\frac1{\Leb(B_n)}\int_{B_n} F(\tau_{x+y}P)\,dy
=
\frac1{\Leb(B_n)}\int_{x+B_n} F(\tau_zP)\,dz.
\]
Therefore
\[
\left|
\mathbf E\left[F(\tau_x\P_n)\right]
-
\mathbf E\left[F(\P_n)\right]
\right|
\le
2\|F\|_\infty\,
\frac{\Leb\left((x+B_n)\Delta B_n\right)}{\Leb(B_n)}\xrightarrow[n\to\infty]{}0
\qquad\text{for every }x\in\mathbb R^d.
\]
Passing to the convergent subsequence $n_k$ and using weak convergence of $\P_{n_k}$, we obtain
\[
\mathbf E\left[F(\tau_x\P)\right]
=
\mathbf E\left[F(\P)\right]
\qquad\text{for all }x\in\mathbb R^d.
\]
Hence $\tau _{ x}\P \stackrel{(d)}{=}  \P$  i.e. $\P$ is stationary.

\smallskip
\noindent
\emph{Step 3:   pure point spectral measure.}
Let
\[
\psi_n(\xi):=\frac1{\Leb(B_n)}\int_{B_n} e^{-2\pi i\langle x,\xi\rangle}\,dx
=
\mathbf E\left[e^{-2\pi i\langle X_n,\xi\rangle}\right].
\]
Then $|\psi_n(\xi)|\le1$, $\psi_n(0)=1$, and by the Riemann--Lebesgue lemma,
\[
\psi_n(\xi)\xrightarrow[n\to\infty]{}0
\qquad\text{for every }\xi\ne0.
\]

Fix $\varphi\in C_c^\infty(\mathbb R^d)$. Since $\widehat P=\sum_{s\in S} a_s\delta_s$ is tempered and discrete $ \hat \varphi $ has fast decay, the series below converge absolutely:
\[
\P_n(\varphi)
=
\sum_{s\in S} a_s\,\widehat\varphi(-s)\,e^{-2\pi i\langle X_n,s\rangle}.
\]
Therefore
\[
\mathbf E[\P_n(\varphi)]
=
\sum_{s\in S} a_s\,\widehat\varphi(-s)\,\psi_n(s),
\]
and
\[
\mathbf E\left[|\P_n(\varphi)|^2\right]
=
\sum_{s,t\in S}
a_s\overline{a_t}\,
\widehat\varphi(-s)\overline{\widehat\varphi(-t)}\,\psi_n(s-t).
\]
Hence
\[
 \textrm{Var}(\P_n(\varphi))
=
\sum_{s,t\in S}
a_s\overline{a_t}\,
\widehat\varphi(-s)\overline{\widehat\varphi(-t)}
\left(\psi_n(s-t)-\psi_n(s)\overline{\psi_n(t)}\right).
\]
Since
\[
\sum_{s,t\in S}|a_s|\,|a_t|\,|\widehat\varphi(-s)|\,|\widehat\varphi(-t)|<\infty
\]
and $|\psi_n|\le1$, dominated convergence applies. Using
\[
\psi_n(\xi)\to \mathbf 1_{\{\xi=0\}},
\]
we get
\[
 \textrm{Var}(\P_n(\varphi))
\longrightarrow
\sum_{s\in S}|a_s|^2\,|\widehat\varphi(-s)|^2
-
|a_0|^2\,|\widehat\varphi(0)|^2.
\]
On the other hand, by translation boundedness, the random variables $\P_n(\varphi)$ are uniformly bounded, and since $\P_{n_k}\Rightarrow\P$,
\[
 \textrm{Var}(\P_{n_k}(\varphi))\to  \textrm{Var}(\P(\varphi)).
\]
Therefore
\[
 \textrm{Var}(\P(\varphi))
=
\sum_{s\in S}|a_s|^2\,|\widehat\varphi(-s)|^2
-
|a_0|^2\,|\widehat\varphi(0)|^2.
\]
This means   that the spectral measure is atomic
  \end{proof}

  }

 {  Let us now turn towards the proof of Theorem  \ref{thm:quasi}, by giving an actually stronger statement.}
The result looks like the conclusion of Proposition  \ref{prop:suff-DEZ-all}, but the latter cannot be applied  because admissible spectral  atomic measures are not necessarily dominated by tensor products of 1D admissible atomic measures. We instead  apply the result to a discretised version of the quasicrystal.
Denote $ C_{\k,t }=t (\k+[0,1)^{d}),t>0,\k\in \mathbb{Z} ^{ d},$ and  define the discrete field $
\M_{t }(\k):=\M(C_{\k,t}). $ The  $ \mathbb{R}^{ d}$-stationarity of  $ \M$ transfers as a $ \mathbb{Z} ^{ d}$-stationarity for $ \M_{ t}$.
We shall prove that for each $ t >0$ (even ``large''), this field is interpolable from the orthant, and therefore that $ \M$ is also interpolable from $ \mathbb{R}_{ -}^{ d}$ in the continuous space. Denote by $ \S_{ t }$ the spectral measure of $ \M_{ t }$ on $  \mathbb  T  ^{ d}.$  

\begin{lemma}
\label{lm:discrete-quasi}For $ \S$ purely atomic, $t>0, \S_{ t }$ is purely atomic. \end{lemma}

\begin{lemma}
\label{lm:no-density-torus}
Let $ \S'$ a purely-atomic non-negative measure on $  \mathbb  T  ^{ d}$. Let $ \varphi \in L^{ 2}(\S')$ whose spectrum is contained in $ (\mathbb{Z} _{  + }^{ d})^{ c}$, then $ \varphi  \equiv 0$.

\end{lemma}

\begin{remark}
Lemma  \ref{lm:no-density-torus} is the main ingredient, it is really about polynomial approximation. It implies that any $ \varphi\in L^{ 2}(\S') $  is approximable by polynomials in $ L^{ 2}(\S').$ We state it on $ (\mathbb{Z} _{  + }^{ d})^{ c}$ instead of $ (\mathbb{Z} _{ -}^{ d})^{ c}$ for some notational simplification.
\end{remark}

\begin{proof}[Proof of Theorem  \ref{thm:quasi}]
Let us first prove  the result for the cone $  \mathbb{R}_{  + }^{ d}$. 
Combining Lemmata \ref{lm:discrete-quasi} and \ref{lm:no-density-torus} gives with Theorem  \ref{thm:general-rigidity} that $ \M_{ t }$ is interpolable from the orthant, i.e. rigid on $ (\mathbb{Z} _{  + }^{ d})^{ c}.$

 Now let $A = (\mathbb{R}_{ + }^{ d})^{ c}, C\subset A$ bounded. Local square integrability yields that $ \M(C)$ is $ L^{ 2}$-approximable by finite linear  combinations of the $ \M(C_{ \k,t }),\k\in \mathbb{Z} ^{ d}\cap A,t >0$. For each $ \k\in \mathbb{Z} ^{ d}\cap A,t >0$, $ \M  (C_{ \k,t })  $ is a $ L^{ 2}$ linear statistic of $ \M_{ t }$, hence by what we just showed, it is $ L^{ 2}$-approximable by linear statistics $ \M_{ t }( f)$ for $ f$ supported by $ \mathbb{Z} _{  + }^{ d}$, and such variables can also be viewed as linear statistics of $ \M$ supported by $ \mathbb{R}_{  + }^{ d}$.  We therefore showed that through successive approximations, $ \M(C)$ is approximable by linear statistics supported by $ \mathbb{R}_{  + }^{ d}$, which concludes the first part.
 
 Let us now give the proof for a cone $ B$ which is not necessarily $ \mathbb{R}_{  + }^{ d}.$ As a cone with non-empty interior, there is an invertible linear mapping $   \mathsf L:\mathbb{R}^{ d} \to \mathbb{R}^{ d}$ such that $   \mathbb{R}_{  + }^{ d}\subset   \mathsf L(B)$. Let $   \mathsf \M_{   \mathsf L} = \M(   \mathsf L^{ -1}\cdot )$ the pushforward measure, satisfying $ \M_{   \mathsf L}(h) = \M(h(   \mathsf L^{ -1}\cdot ))$ for $ h\in \mathscr C_{c}^{ b}(\mathbb{R}^{ d})$. The spectral measure of $   \M_{    \mathsf L}$ is also purely atomic, hence it is interpolable from $ \mathbb{R}_{  + }^{ d}$.  Therefore $ \M =   (  \M_{   \mathsf L})_{   \mathsf L^{ -1}}$ is interpolable from $   \mathsf L^{ -1} (\mathbb{R}_{  + }^{ d}) \subset B.$

\end{proof}

\begin{remark}
We in fact prove a stronger form of maximal rigidity: the quasicrystal is ``discretely'' MR, in the sense that each  $ t -$discretisation $ \M_{t }$ is interpolable from $ \mathbb{Z} _{  + }^{ d},t >0$. Surprisingly, it also holds for $ t $ ``large''. Continuous MR does not imply discrete MR in general, as we see that a deep zero of $ \S$ does not imply that each $ \S_{t }$ has a deep zero. Stealthy measures at section \ref{sec:stealthy} are in general not discretely MR.
\end{remark}

\begin{remark}
A line of results triggered by a conjecture of Lagarias  \cite{Lagarias} explores whether  a (uniformly discrete) quasicrystal  is a union of periodic lattice combs, as in Example \ref{ex:combs}. See in particular the  result by Lev and Olevskii   \cite{LevOlev}, showing that the answer is positive if the spectrum is locally finite, and the negative example of Kurasov and Sarnak  \cite{KurSarnak}, and  references therein.

We consider in this work random stationary models, but many deterministic quasicrystals can be turned into a stationary model by exploring the orbit ( \cite{HartBjo}). In this respect, the above result of interpolation from convex cones applies to such models a.s. We also develop at Section \ref{sec:strong-rigid} below the concept of strong rigidity: a quasicrystal for which a stronger approximability result than Lemma  \ref{lm:no-density-torus} holds would be periodic, see
Propositions \ref{prop:strong-rigid},  \ref{prop:periodic}. 

\end{remark}

Regarding interpolability, the next question is therefore what is the smallest (class of sets) $ B\subset \mathbb{R}^{ d}$ such that a random measure with atomic spectrum can always be determined from its restriction to $ B$.  We can already discard the case of an infinite constant-width band  by taking the union of infinitely many shifted lattices with increasing mesh sizes:

\begin{example}[Dirac combs]
\label{ex:combs}
 Let $ a_{i}>0,i\geqslant 1$ with $ \sum_{i}a_{i}^{-d}<\infty $ and $  U_{i},i\geqslant 1$, iid uniform variables in $ [0,1]^{ d}$. Define the random stationary measure $ \M=\sum_{i}\M_{i} $ with
\begin{align*}
\M _{i} = \sum_{\k\in \mathbb{Z}^{ d} }\delta _{a_{i}(U_{i} + \k)}.
\end{align*}

The intensity is $ \sum_{i}a_{i}^{-d}$ and $ \M$ is locally square integrable.
The spectral measure is indeed atomic as the Poisson summation formula easily yields that the spectral measure of $ \M_{i}$ is $ \S_{ i} = \sum_{\k\in \mathbb{Z} ^{d}  \setminus \{0\}}a_{i}^{-d}\delta _{a_{i}^{-1}\k}$.
Then for $ c>0$, on   $ B = [0,c]\times  \mathbb R ,$ it is impossible to infer the value of the $ U_{i}$ for $ i$ such that $ \M_{i} (B) =0  $, and this occurs a.s. for infinitely many $ {i}$, therefore $ \M$ is not perfectly interpolable from $ B.$  
\end{example} 

The previous example is not a quasicrystal as it is not uniformly discrete (one can find pair of points at an arbitrarily small positive distance).
It leaves us with the following questions:

 \begin{question}
Let $ \varphi :\mathbb{R}_{ + }\to \mathbb{R}_{ + }$ such that $ \varphi  (x)\to \infty $ as $ x\to \infty $, with $ \varphi (x) = o(x).$ Is a random stationary measure with atomic spectrum always perfectly interpolable from 
\begin{align*}
B= \{(x,t):0\leqslant t\leqslant \varphi (x)\}?
\end{align*}
\end{question}

\subsubsection{Proof of Lemma  \ref{lm:discrete-quasi}}  

We compute $  \textrm{Var}\left(\M_{ t }(f)\right)$ for $ f:\mathbb{Z} ^{d}\to  \mathbb C $ with bounded support. Let
\begin{align*}
f_{t }(x)=\sum_{\k\in \mathbb{Z} ^{ d}}1\{ x\in C_{\k,t }\} f(\k),x\in \mathbb{R}^{ d}.
\end{align*}
We have a.s. $ \M_{t }(f)=\M(f_{t })$, hence by   \eqref{eq:phase}\begin{align*}
 \textrm{Var}\left(\M_{t }(f)\right)=&  \textrm{Var}\left(\M(f_{ t })\right)\\=&(2\pi )^{ -d}\int_{\mathbb{R}^{d}} | \widehat{ f_{t }}(v) | ^{2}\S(dv).
\end{align*}
Define 
\begin{align}
\label{eq:def-J}
 J(s) = \int_{[0,1]^{ d}}e^{ -is\cdot w}dw =  \prod_{l = 1}^{ d}\sinc(  s_{ l}),s = (s_{ 1},\dots ,s_{ d})\in \mathbb{R}^{ d}.
\end{align}
Then, for $ s\in  \mathbb  R  ^{ d}$
\begin{align*}
\widehat{ f_{t }}(s)=&\int_{\mathbb{R}^{ d}} f_{ t }(w)e^{ -i s\cdot w}dw = \sum_{\k\in \mathbb{Z} ^{ d}}f(\k)e^{- is t\cdot \k} t ^{ d}\int_{[0,1]^{ d}}e^{-is t \cdot w}dw= t ^{ d}\hat f(e^{ its })J(t s ) 
\end{align*}
recalling that $ \hat f$ is defined on $  \mathbb  T  ^{ d}$.  Since $ \S$ is atomic locally finite, write $ \S = \sum_{j}a_{ j}\delta _{ s ^{ j}}$ with $ a_{ j}>0,s^{ j}\in \mathbb{R}^{ d}$. We have
\begin{align*}
 \textrm{Var}\left(\M_{t }(f)\right)=&(2\pi )^{ -d}
 \int_{\mathbb{R}^{d}} | \hat f(e^{ i ts}) | ^{2}t ^{2d} | J({ t s}) | ^{2}\S(ds )\\
  = &t ^{ 2d}(2\pi )^{ -d}\sum_{j}a_{ j} | \hat f(e^{ it s ^{ j}  }) J({ t s ^{ j}}) | ^{ 2} 
\end{align*}
whch means by  \eqref{eq:phase}  that $\M_{ t }$ has the spectral measure
\begin{align}
 \label{eq:S-eps}
 \S_{t } =  t ^{2d}\sum_{j}a_{ j} | J( { t s^{ j}}) | ^{ 2}\delta _{ e^{i t s ^{ j}}}.
\end{align}
 It is indeed also purely atomic. It is finite because it is the spectral measure of the $ L^{ 2}_{ \text{\rm{loc}}}$ measure $ \M_{ t}$, i.e. $ \S_{ t }(  \mathbb  T  ^{ d}) =  \textrm{Var}\left(\M_{ t }(0)\right)$.

\subsubsection{Proof of Lemma \ref{lm:no-density-torus}}

This lemma likely exists in the literature, but we could not locate precisely this multi-dimensional version  \draft{ check Rudin, Function Theory in Polydiscs, Chap.?1-2.
The same argument extends: analytic polynomials in $(z_1,\dots,z_d)$ are dense in $(L^2(\mu))$ for any discrete measure $(\mu)$ on $(T^d)$.
The proof uses the Stone?Weierstrass theorem on the compact set $({x_j}\subset T^d)$.
}.

Let $ \S' = \sum_{j}a_{ j}\delta _{ u ^{ j}}$ a purely atomic non-negative finite measure on $  \mathbb  T  ^{ d}$ where the $ u^{ j}$ are distinct.  We write $ u^{ j} = (u^{ j}_{ 1},\dots ,u^{ j}_{ d})$ where $ u^{ j}_{ k} = e^{ it^{ j}_{ k}}$   for some $ t^{ j}_{ k} \in [0,2\pi ).$  Let $ \varphi \in L^{ 2}(\S')$ whose spectrum is contained in $( \mathbb{Z} _{  + }^{ d})^{ c}$, i.e. such that
\begin{align}
\label{eq:ortho-poly}
 \int_{  \mathbb  T  ^{ d}}\varphi (u)u^{ \m}\S'(du) = 0, \m\in \mathbb{Z}_{  + } ^{ d} .
\end{align} 
We must prove $ \varphi \equiv 0.$
 Define $ c_{ j} = \varphi (u^{ j})a_{ j}$ and  the complex-valued measure
\begin{align*}
 \mu  = \sum_{j}c_{ j}\delta _{ u^{ j}}
\end{align*}which is finite by Cauchy-Schwarz:
\begin{align*}
 \sum_{j} | c_{ j} |  = \sum_{j} | \varphi (j) | a_{ j}\leqslant \sqrt{\sum_{j}a_{ j}\sum_{j} | \varphi (u^{ j}) | ^{ 2}a_{ j}} = \sqrt{\S'(  \mathbb  T  ^{ d})}\|\varphi \|_{ L^{ 2}(\S')}<\infty .
\end{align*}
Also, by  \eqref{eq:ortho-poly}, $ \mu $ has vanishing Fourier coefficients with positive coordinates. 
 Let us prove that it implies $ \mu \equiv 0.$
Define $  \mathbb D_{ r} = \{z_{ 1}\in  \mathbb  C  : |  z_{ 1} | <r\},r\in (0,1]$ and $  \mathbb D =  \mathbb D_{ 1}$. Recall the series expansion
\begin{align*}
 \frac{ 1}{1-z} = \sum_{n\geqslant 0}z^{ n},z\in  \mathbb D.
\end{align*}
 We introduce  for $ z = (z_{ 1},\dots ,z_{ d})\in   \mathbb D^{ d}$   
\begin{align*}
 F(z_{ 1},\dots ,z_{ d}) = &\int_{  \mathbb  T  ^{ d}}\frac{ \mu (du)}{(1-z_{ 1} {u_{ 1}})\dots (1-z_{ d} { u_{ d}})}\\
  = &\sum_{j}c_{ j}\frac{ 1}{(1-z_{ 1}{ u^{ j}_{ 1}})\dots (1-z_{ d}{ u^{ j}_{ d}})}\\
  =& \sum_{j}c_{ j}\sum_{n_{ 1},\dots ,n_{ d}\geqslant 0}z_{ 1}^{ n_{ 1}} (u_{ 1}^{ j})^{ n_{ 1}}\dots  z_{ d}^{ n_{ d}} (u_{ d}^{ j})^{ n_{ d}}.
\end{align*}
For $ r<1$, on the domain $   \mathbb D_{ r}^{ d}$, we can switch all summations because 
\begin{align*}
 \sum_{j} | c_{ j} | \sum_{n_{ 1}\geqslant 0}r^{ n_{ 1} }\dots \sum_{n_{ d}\geqslant 0}r^{ n_{ d}}<\infty .
\end{align*}
Therefore, for $ z\in  \mathbb D^{ d}$
\begin{align*}
 F(z) = \sum_{n_{ 1},\dots ,n_{ d}\geqslant 0}(\sum_{j}c_{ j} (u_{ 1}^{ j})^{ n_{ 1}}\dots (u_{ d}^{ j})^{ n_{ d}})z_{ 1}^{ n_{ 1}}\dots z_{ d}^{ n_{ d}} = \sum_{\m\in \mathbb{Z} _{  + }^{ d}}z^{ \m}\int_{}u^{ \m}d\mu 
\end{align*}
and the internal integral is a positive Fourier coefficient of $ \mu $, it vanishes by assumption. Hence $ F\equiv 0$ on the domain $  \mathbb D^{ d}$.
To formally show that it implies that each $ c_{ k} = 0$, we must study the continuity at the boundary. Define  on $  \mathbb D^{ d}$
\begin{align*}
 G_{ k}(z_{ 1},\dots ,z_{ d}) = &(1-z_{ 1}u_{ 1}^{ k})\dots (1-z_{ d}u_{ d}^{ k})F(z_{ 1},\dots ,z_{ d})\\
  =& c_{ k} + \sum_{j\neq k}c_{ j}\frac{ (1-z_{ 1}u_{ 1}^{ k})\dots (1-z_{ d}u_{ d}^{ k})}{(1-z_{ 1}u_{ 1}^{ j})\dots (1-z_{ d}u_{ d}^{ j})},
\end{align*}
this function clearly also vanishes on $  \mathbb D.$ Since on the torus $ u_{ i}^{ j}  {  \bar  u_{ i}^{ j}}  = 1$, each summand of order $ j\neq k$ vanishes when $ z = (z_{ 1},\dots ,z_{ d})\to   \overline{   u^{ k}} = (  \overline{ u_{ 1}^{ k}} ,\dots ,  \overline{ u_{ d}^{ k}} ).$ We would like to show 
\begin{align*}
 c_{ k} = \lim_{ z\in  \mathbb D^{ d},z\to  \bar u^{ k}}G_{ k}(z_{ 1},\dots ,z_{ d}) = 0.
\end{align*}
It suffices to study radial limits with 
\begin{align*}
 g_{ k}(r_{ 1},\dots ,r_{ d}): = G_{ k}(r_{ 1}  \bar u_{ 1}^{ k},\dots ,r_{ d}  \bar u_{ d}^{ k}) = c_{ k} + \sum_{j\neq k}c_{ j}\frac{ (1-r_{ 1})\dots (1-r_{ d})}{(1-r_{ 1}  \bar u_{ 1}^{ k} u_{ 1}^{ j})\dots (1-r_{ d}   \bar u_{ d}^{ k} u_{ d}^{ j}))}.
\end{align*}
For $ r<1$, the point of the form $ re^{ i\theta }$ the closest to $ 1$ is $ r$, hence 
\begin{align*}
  | 1-r   \bar u_{ i}^{ k} u_{ i}^{ j}|  =  | 1-re^{ i(t_{ i}^{ j}-t_{ i}^{ k})} | > | 1-r | ,
\end{align*}
and each summand is dominated by  $
 |  c_{ j} | \times 1 $,
recall that the $ c_{ j}$ are summable. By Lebesgue's theorem, this concludes the proof that $ \mu  = 0$, hence $ \varphi \equiv 0.$

\section{Strong rigidity and periodicity}
\label{sec:periodicity}

\subsection{Periodicity of discrete fields}
  
Borichev, Nishry and Sodin \cite{BNS}  show that an  integer-valued time series $ \X$  with a spectral gap  is necessarily periodic. This result was later refined in \cite{BSW18} using variants of Szeg\"o's theorem. Both these works are uni-dimensional, and \cite{BSW18} mentions that  {\it an intriguing question is to extend their main results to stationary processes on $ \mathbb{Z} ^{d}$ with $ d\geqslant 2$.} By developing a connection with the question of interpolability from 
 convex cones, we shall prove here the following generalisation: an integer valued   stationary random field $ \X$ is periodic if its spectrum is simply connected. Recall that a path-connected subset $ S\subset  \mathbb  T  ^{ d}$ is simply connected if any  loop, i.e. $  \mathbb  T  ^{ d}$-continuous $ \gamma :[0,1]\to S$ with $ \gamma (0) = \gamma (1)$, can be continuously deformed into a point within $ S$, i.e. there exists a continuous mapping $ T :[0,1]^{ 2} \to S$ such that $ T(0,\cdot ) = \gamma ,T(a,0) = T(a,1)$  and $ \# T(1,[0,1]) = 1.$\\

 We emphasize that in this section, contrary to the rest of the paper, random fields $ \X = \{\X_{\m},\m\in \mathbb{Z} ^{d}\}$ are strongly stationary, i.e. $ \tau _{\m}\X$ and $ \X$ have the same law for any $ \m\in \mathbb{Z} ^{d}.$ In this case, say that $ X$ is  {\it ergodic} if 
$ \mathbf P (\X\in \Omega )\in \{0,1\}$ for any event $ \Omega $ invariant under translation, i.e. such that for every $ \m\in \mathbb{Z} ^{ d}$, $ 
 \Omega  = \{\tau _{ m}\omega :\omega \in \Omega \}.$

Say that a deterministic sequence $ x_{\m};\m\in \mathbb{Z} ^{d}$ is  $ N$-periodic  for some $ N\in \mathbb{Z} _{  + }^{d}$  if for all $\m\in \mathbb{Z}^{d} ,x_{\m + N} = x_{\m}$. By abuse of notation, for $ N\in \mathbb{Z} _{  + }$, say that it is $ N$-periodic if it is $ (N,\dots ,N)$-periodic. Call  {\it uniformly discrete}  a set $  \U\subset  \mathbb C $  such that 
\begin{align*}
\delta _\U: = \inf_{z,z'\in  \U,z\neq z'} | z-z' | >0.
\end{align*}

\begin{theorem}
\label{thm:BSW-d}
 Let $ \X$ a strongly stationary field on $ \mathbb{Z} ^{d}$ taking values in a uniformly discrete set $  \U$. Assume the support of its spectral measure $ \S$ is contained in a simply connected set  of $  \mathbb  T  ^{ d}$. Then $ \X$ is a.s. periodic, with a possibly random period. In particular, $ \S$ is a purely atomic measure.

If $ \X$ is furthermore assumed to be ergodic, it is a.s. $ N$-periodic for a deterministic period $ N\in \mathbb{N}$, and $ \S$'s support is finite.
 \end{theorem}

The proof is at Section  \ref{sec:prf-periodic}.
 This theorem hence shows that there is a strong constraint on the possible shapes for the spectra of random fields taking values in a uniformly discrete set. 
The fundamental concept behind the scenes is that of polynomial approximation on $ \S$'s support.
For notational simplification we consider here the positive orthant $ \Q =  \mathbb{Z} _{  + }^{ d}$.
For $ B\subset \Q$, let $  \mathscr  P_{ B}$ be the class of polynomials of the form
\begin{align*}
 P(z) = \sum_{\m\in B}h_{ \m}z^{ \m},z\in  \mathbb C ^{ d}.
\end{align*}
Say that $ \hat B\subset  \mathbb  T  ^{ d}$ is  {\it approximable} if its indicator function is uniformly approximable by   polynomials of  $  \mathscr  P_{ \Q}$, i.e.\begin{align*}
 \inf_{ P\in  \mathscr  P_{ \Q}}\|1-P  
 \|_{ \hat B}= 0
\end{align*} 
where $ \|Q\|_{ \hat B} = \sup_{ z\in \hat B} | Q(z) | .$
 Say furthermore that $ \hat B$ is isotropically approximable if for every permutation $ \sigma $ of $ \{1,\dots ,d\}$, the corresponding permutated set $ \hat B^{ \sigma }: =  \{(x_{ \sigma (1)},\dots ,x_{ \sigma (d)}):x\in \hat B\}$ is approximable.
We show below with Proposition \ref{prop:corridors} that if $ \hat B$ has straight  gaps across each coordinate, $ \hat B$ is isotropically approximable.

 The general assumption is actually that the indicator function of the support should be approximable by polynomials, because it means that its complement forms a uniqueness pair with $ \Q$, see Section  \ref{sec:strong-rigid} and Proposition  \ref{prop:strong-rigid} below. We exploit a result of Stolzenberg  \cite[(vi), p.262]{Stolzenberg} that says that simply connected sets are polynomially convex, and therefore approximable by  the Oka-Weil theorem  \draft{ ref}.
 
 Let us give a simple example of an approximable (and simply connected) subset. We provide for illustrating purposes at Section  \ref{sec:prf-periodic}     a direct proof of approximability not requiring  \cite{Stolzenberg}.
 \begin{proposition}\label{prop:corridors}
 Let $  I_{ 1},\dots ,I_{ d}$  non-empty open sets of $\mathbb  T  $, then 
\begin{align*}
 \hat B: = ( \mathbb  T   \setminus I_{ 1})\times \dots \times ( \mathbb  T   \setminus I_{ d})
\end{align*}is simply connected and isotropically approximable.
 \end{proposition} 
 
    See Table  \ref{table:1}-(d) for another example, and Table \ref{table:1}-(e) for  a counter-example.
It might be possible to replace the gap assumption by that of having a deep zero in a  sense stronger than  \eqref{eq:dez-torus}, such as in   \cite[Theorem 8]{BSW18}.

\subsubsection*{Acknowledgment} I am indebted to M. Sodin, who provided the reference  \cite{Stolzenberg},  and the connection with Oka-Weil theorem, allowing to state Theorem  \ref{thm:BSW-d} with the simply connected hypothesis.  He also pointed to many items of the literature I was unaware of.

 \subsection{Counter-example and discussion}
 
 Simple connectedness of the spectrum implies to cut the torus along each coordinate,  which means there should be a hole along each dimension, but this hole can vary with the level, see Table \ref{table:1}-(d). This condition is not necessary:
\begin{proposition}
There exists  $ S\subset  \mathbb  T  ^{ d}$ that is approximable and that is not a subset of a simply connected subset.
\end{proposition}  

The proof is at Section  \ref{sec:counter-ex-SC}, along with the graphical illustration Figure  \ref{fig:counter-ex}.
 Looking back at Proposition  \ref{prop:separable}, what seems necessary is that there should not be a  {\it straight  line} in the spectrum. Is it also sufficient?
Finding a necessary and sufficient condition for support approximability is likely a  difficult question of complex analysis.

 \subsection{Strong interpolability}  
 \label{sec:strong-rigid}
 
The proof requires a
stronger form of interpolability for $ \X$, implicit in   \cite{BSW18}.  Let $ \Q_{n}:=\Q\cap [1,n]^{d}$. ``Classical'' perfect interpolability (as treated in the rest of the paper)  means that the best linear approximation of $ \X_{ 0} $ of the form $ \hat \X_{ 0}:=\X (h):=\sum_{\m\in \Q_{n}}h_{\m}\X_{ \m}$, has a vanishing error as $ n\to \infty $, i.e.    \begin{align*}
e_{n}(\S): = \inf_{h \in   \mathbb C ^{  \Q_{ n}}}\mathbb{E}( | \X _{ 0}-\X (h) | ^{2}) =   \inf_{ P\in  \mathscr  P_{ \Q _n}}\int_{  \mathbb  T  ^{ d}} | 1-P | ^{ 2}d\S\xrightarrow[ n\to \infty ]{}0.
\end{align*}  We say that  $ \X$ is  {\it strongly interpolable} from $ \Q$, if furthermore
\begin{align*}
\sum_{n}n^{d-1}e_{n}(\S)<\infty .
\end{align*}
 In comparison, for instance, Lemma \ref{lm:no-density-torus} is a result of $ L^{ 2}$-polynomial approximation.
 Therefore Theorem \ref{thm:BSW-d} is a consequence of  the two  next propositions.

\begin{proposition}
\label{prop:strong-rigid}
Assume   $  { \rm supp}(\S)    $  is  approximable. Then   $ \X$ is strongly interpolable from $  \Q$. 
\end{proposition}

By permutation of the coordinates, if $  { \rm supp}(\S)   $ is isotropically approximable, then $ \X$ is strongly interpolable from any orthant, i.e. any set obtained by finitely many translations and reflections applied to $ \Q$, such as in Proposition  \ref{prop:corridors}.

\begin{proposition}
\label{prop:periodic}
A  strongly stationary process $ \X$ taking values in a UD set which is strongly interpolable from every orthant is a.s. periodic. If $ \X$ is furthermore ergodic, the period can be chosen deterministically.
\end{proposition}

\subsection{Proofs}  
  \label{sec:prf-periodic}

\begin{proof}[Proof of Proposition \ref{prop:corridors}]
In dimension $ d = 1$, the result is a consequence of the classical Runge or Weierstrass theorems ( \cite[Th. 13.7]{Rudin}): for each $ 1\leqslant i\leqslant d$, $  \mathbb  T   \setminus I_{i}$ is approximable, i.e. 
$$  \inf_{Q_{i}\in  \mathscr  P_{\mathbb{N}^{*}}}  \|1-Q_{i}\|_{  \mathbb  T   \setminus I_{i}}=0.$$
In higher dimensions, it is a consequence of the fact that the tensor product of polynomially convex sets is polynomially convex. Let us give a proof for completeness. By isotropy, it suffices to prove that $ \hat B$ is approximable.
Remark  that it changes nothing to require that approximating polynomials are bounded by $ 2$ or another constant on $ \hat B.$ 
Then, since for $ Q_{ i}\in  \mathscr  P_{ \mathbb{N}^{ *}}$, $ Q_{ 1}\otimes \dots \otimes Q_{ d}\in  \mathscr  P_{\Q},$
\begin{align*}
 \inf_{ P\in  \mathscr  P_{\Q},\|P\|_{  \hat B}\leqslant 2^{ d} } \| 1-P  \|_{ \hat B} \leqslant &\inf_{ \stackrel{ Q_{i}\in  \mathscr  P_{\mathbb{N}^{*}},1\leqslant i\leqslant d}{\|Q_{ i}\|_{  \mathbb  T   \setminus I_{ i}}\leqslant 2 }   }\sup_{ z_{ i}\in  \mathbb  T   \setminus I_{i} } | 1-Q_{1}(z_{ 1})\dots Q_{d}(z_{ d}) | \\
  \leqslant&  \inf_{ \stackrel{ Q_{i}\in  \mathscr  P_{\mathbb{N}^{*}},1\leqslant i\leqslant d}{\|Q_{ i}\|_{  \mathbb  T   \setminus I_{ i}}\leqslant 2 }   } \sup_{ z_{ i}\in  \mathbb  T   \setminus I_{i}}\sum_{j=1}^{d}   \prod_{1\leqslant i\leqslant j} | Q_{i}(z_{i}) |   \prod_{i=j+1}^{d} | 1-Q_{i}(z_{ i}) | ^{2}\\ 
   \leqslant& 2^{ d}\inf_{ \stackrel{ Q_{i}\in  \mathscr  P_{\mathbb{N}^{*}},1\leqslant i\leqslant d}{\|Q_{ i}\|_{  \mathbb  T   \setminus I_{ i}}\leqslant 2 }   } \sup_{ z_{ i}\in  \mathbb  T   \setminus I_{i}}\sum_{I\subsetneq  \llbracket d \rrbracket}      \prod_{i\notin I} | 1-Q_{i}(z_{ i}) | ^{2}\\
  \leqslant &c_{d} \sup_{ I \subsetneq \llbracket d \rrbracket}  \prod_{i\in I}\inf_{ \stackrel{ Q_{i}\in  \mathscr  P_{\mathbb{N}^{*}} }{\|Q_{ i}\|_{  \mathbb  T   \setminus I_{ i}}\leqslant 2 }   }\sup_{ z_{ i}\in  \mathbb  T   \setminus I_{i}} | 1-Q(z_{i}) |   =0.
\end{align*}
\end{proof}

\begin{proof}[Proof of Proposition  \ref{prop:strong-rigid}]
  By the approximability assumption, there is $ k$ and a polynomial $ Q(z) =1- \sum_{\m\in \Q_{k}}h_{\m,1}z^{\m}:  \mathbb C^{d} \to  \mathbb C $ bounded by $ 1/2$ on $    {S: =  \rm supp}(\S)     $. For $ n\geqslant 1,$ let $ h_{\m,n}$ the coefficients defined by 
\begin{align*}
Q(z)^{n}=1-\sum_{\m\in \Q_ {kn}}h_{\m,n}z^{\m}.
\end{align*}    
Then we have 
\begin{align*}
e_{kn}(\S)
=&\inf_{h\in \Q_{kn}^ \mathbb C }\int_{  { \rm supp}(\S)   } | 1-\sum_{\m\in \Q _{kn}}h_{\m}u^{ \m} | ^{2}d\S(u)\\
\leqslant & \int_{  { \rm supp}(\S)   }  |  Q(u) | ^{2n}d\S(u)\\
\leqslant &4^{-n}\S(  \mathbb  T  ^{d}).
\end{align*}
 
Since $ n \mapsto e_{n}(\S)$ is non-increasing, this implies strong interpolability
\begin{align*}
\sum_{n}n^{d-1}e_{n}(\S) <\infty .
\end{align*}
\end{proof} \begin{proof}[Proof of Proposition   \ref{prop:periodic} ]

The proof relies on the following lemma, proved later. For $ X\in ( { \mathbb{Z} ^{ d}})^{  \mathbb C },B\subset \mathbb{Z} ^{ d}$, let $ X_{ B} = (X_{ \m},\m\in B).$
\begin{lemma}Denote by $ C_{n} =(\mathbb{Z} \cap [-n,n])^{d},D_{n} : = C_{n + 1}  \setminus C_{n}$. There are deterministic mappings $ \varphi _{n}: \U^{ C_{n}}\to \U^{  D_{n}} ,n\in \mathbb{N}$ such that
\begin{align*}
\mathbb{P}(\X_{D_{n}}\neq \varphi _{n}(\X_{C_{n}}))\leqslant cn^{d-1}e_{[n/2]-1}(\S).
\end{align*}
\end{lemma}
Borel-Cantelli lemma and the strong interpolability assumption therefore imply that there is a  random variable $ N_{0}\in \mathbb{N}$ such that a.s., for $ N\geqslant N_{0}$, $ \X_{D_{N}}=\varphi _{N}(\X_{C_{N}}).$
By recursion, $ \X$ is then completely determined by $ \X_{C_{N_{0}}}$.
More precisely,  we can define for each $ n_{ 0}\geqslant 1$ the mapping $ \psi _{n_{0}}: \U^{ C_{ n_{ 0}}}\to \U^{ \mathbb{Z} ^{ d}}$ in the following way: for $X\in \U^{ C_{ n_{ 0}}}, \m\in C_{ n_{ 0}}, \psi _{ n_{ 0}}(X)_{ \m} := X_{ \m}$, and for $ \m\notin C_{ n_{ 0}}$, 
\begin{align*}  
\psi _{ n_{ 0}}(X) _{ \m}: = \varphi _{ n}(\varphi _{ n-1}(...(\varphi _{ n_{ 0}  }(X))\dots )_{ \m}
\end{align*}
where $ n$ is such that $ \m\in C_{ n},$ so that a.s. $ \X = \psi _{ N_{ 0}}(\X_{ C_{ N_{ 0}}})$.
Before proving the lemma, let us conclude the proof by adapting the argument from  
\cite{BSW18}. Since $  \U$ is uniformly discrete, it is countable, and for each finite $ C\subset \mathbb{Z} ^{d}, {\bf	U}^{ C}$ is countable. For each $ n_{0}\in \mathbb{N}$, call $ \omega _{i}^{n_{0}},i\geqslant 1$ the elements of $  {  \U}^{ C_{ n_{ 0}}}.$ By the previous lemma,  a.s.,
\begin{align*}
\X = \psi  _{N_{0}}(\X_{C_{N_{0}}}) = \psi  _{N_{0}}(\omega _{I}^{N_{0}})
\end{align*}where $ I$ is a random integer. It means that $ \X$ takes values in the countable set 
\begin{align*}
\Omega : = \{\psi  _{n_{0}}(\omega ^{n_{0}}_{i});i,n_{0}\in \mathbb{N}\}.
\end{align*}
We partition $ \Omega $ in finite sets $ \Omega _{j},j\geqslant 1$ where, for two $ \omega ,\omega '\in \Omega $, $ \omega ,\omega '$ belong to the same $ \Omega _{j}$ for some $ j$  if and only  if  they are translates of one another, i.e. $ \omega  = \tau _{\m}\omega '$ for some $ \m\in \mathbb{Z} ^{d}$. We then have by translation invariance   
\begin{align*}
\mathbf P (\X = \omega ) = \mathbf P (\X = \omega ') .\end{align*}
Therefore, calling $ \omega _{j}$ some arbitrary representative of $ \Omega _{j},$ and $ p_{j} = \mathbf P (\X = \omega _{j}),$
\begin{align*}
\mathbf P (\X\in \Omega _{j})  =  | \Omega _{j} | \mathbf P (X = \omega _{j}) =  | \Omega _{ j} | p_{ j}.
\end{align*}
We assume without loss of generality that all $ p_{j}$ are strictly positive, and therefore the $  | \Omega _{j} | $ are finite. Since $ \Omega _{ j}$ is invariant under translations, it implies that there exists $ N_{j}\in (\mathbb{N}_{*})^{d}$ such that for all $ \omega \in \Omega _{j}, \tau _{N_{J}}(\omega ) = \omega .$
 Let $ J$ the random integer such that $ X\in \Omega _{J}$. Then $ t_{N_{J}}\X = \X$, meaning $ \X$ is $ N_{J}$-periodic, which concludes the proof. In the ergodic case, remark that for $ N\in \mathbb{N}^{d}$, $ \{t_{N}\X = \X\}$ is an event invariant under translations, therefore for all $ j$, $ \mathbf P (t_{N_{j}}\X = \X)\in \{0,1\}$, which proves that indeed $ t_{N_{j}}\X = \X$ a.s. for some deterministic $ N_{j}.$

Let us now prove the lemma.
For $ z\in  \mathbb C ,$ denote by $ [z]$ the  element of $  \U$ closest to $ z$, with ties broken in an arbitrary way. Remark that if some $ z'\in   \U$ satisfies $  | z-z' | <\delta_{ \U}/2 $, then $ z'=[z].$ 
By definition of $ e_{ n}(\S)$, there exist coefficients $ h = (h_{\m})\in  \mathbb C ^{  \Q_{n}}$, such that, with Byenaim\'e-Tchebyshev inequality,
\begin{align*}
\mathbb{P}( | \X_{ 0}-\sum_{\m\in \Q_{n} }h_{\m}\X _{ \m} | >\delta_{ \U}/2 )\leqslant 4\frac{\sup_{h\in  {\Q  _{ n}^{  \mathbb C }}}\mathbb{E}( | \X_{ 0}-\X (h) | ^{2})}{\delta _{ \U}^{2}}\leqslant c_{ \U}e_{n}(\S ).
\end{align*}
Since $ \X _{ 0}\in  \U$, we have
\begin{align}
\label{eq:approx-0}
\mathbb{P}(\X_0\neq [\X (h)])\leqslant c_{  {\bf U}}e_{n}(\S ).
\end{align}

This means that the value of $ \X$ at the origin, seen as an ``outer vertex''   of the cube $ \Q_{n}$ can be inferred from $ \X_{\Q_{n}}$ with probability at least $ 1-c_{ \U}e_{ n}(\S)$. By isotropy, the same should be true for other outer vertices  of $ \Q_n$, defined as the $ 2^{d}$ points 
\begin{align*}
V(\Q_{n}) = \{(\eta  _{i}(n + 1)^{d})_{i = 1,\dots ,d},\eta  \in \{0,1\}^{d}\}.
\end{align*}
We call cube of size $ n$ any translate  $C = \m +  \Q_{n}$ for some $ \m\in \mathbb{Z} ^{d}$, and the outer vertices are the 
\begin{align*}
V(C) = \{\m+V(\k),\k\in V(\Q_{n})\}.
\end{align*}

 By symmetry of the assumptions,  \eqref{eq:approx-0} still holds after applying isometries to $ \X$, and we can  state the following: for any cube $ C$ with sidelength $ n$, $ \k\in V(C)$, there is a mapping $ \varphi _{C,\k}:C\to  \U$ such that
\begin{align*}
\mathbf P (\X_{\k}\neq \varphi _{C,\k}(\X_{\C_{n}}))<ce_{n}(\S).
\end{align*}
The important observation to propagate values of some
 $ X\in  \U^{C_{n}}$  to any point $\k\in D_{n}$ is that there is always a cube $ C_{ n,\k}\subset C_{ n}$ of size at least $ [n/2]-1$ such that $ \k\in V(C)$. Hence we predict the value at $ \k$ through
\begin{align*}
\varphi _{n}(\X)_{ \k}  : = \varphi _{ C_{ n,\k},\k}(\X)
\end{align*}
so that indeed $ \varphi _{n}:  \U^{C_{n}}\to  \U^{D_{n}}$ satisfies
\begin{align*}
\mathbf P (\varphi _{n}(\X_{ C_{ n}})\neq \X_{D_{n}}) =&\mathbf P (  \exists  \k\in D_{n}: \varphi _{n}(\X_{ C_{ n}})_{k}\neq \X_{\k})\\
\leqslant &\sum_{\k\in D_{n}}\mathbf P (\X_{\k}\neq \varphi _{C_{ n,\k},\k}(\X_{C_{n}}))\\
\leqslant & c'| D_{n} | e_{[n/2]-1}(  \S),
\end{align*}
by \eqref{eq:approx-0},
which allows to conclude the proof.

\end{proof}

 \subsection{An approximable set which is not a subset of a simply connected set}
 \label{sec:counter-ex-SC}

Let $ I_{ a,b} = \{e^{ i\theta }; a\leqslant \theta \leqslant b\}\subset   \mathbb  T . $
Let $ \eta  = 1/100$. Then define 
\begin{itemize}
\item $   \mathbb  T  ^{ x} =  \mathbb  T   \setminus I_{ x-\eta,x + \eta  }$
\item $ J_{0} = I_{ -\pi /3,\pi /3} $
\item $ J_{ \pi } = I_{ \pi -\pi /3,\pi  + \pi /3}$
\item $ J ^{ 0,\pi } =  \mathbb  T    \setminus (J_{ 0}\cup J_{ \pi })$
\end{itemize}

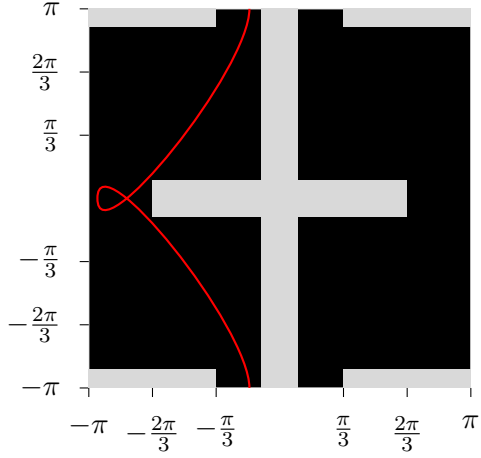
\begin{figure}[h!]
    \centering 

\begin{tikzpicture}[scale=0.8, font=\small]

\def\p{3.1416}      
\def\a{1.0472}      
\def\b{2.0944}      
\def\e{.3}
\fill[black,draw=white] (-\p,-\p) rectangle (\p,\p);


\foreach \x/\lab in {-3.1416/-\pi, -2.0944/-\tfrac{2\pi}{3}, -1.0472/-\tfrac{\pi}{3},
                      1.0472/\tfrac{\pi}{3}, 2.0944/\tfrac{2\pi}{3}, 3.1416/\pi}
  \draw (\x,-\p) -- (\x,-\p-0.15) node[below=3pt]{$\lab$};

\foreach \y/\lab in {-3.1416/-\pi, -2.0944/-\tfrac{2\pi}{3}, -1.0472/-\tfrac{\pi}{3},
                      1.0472/\tfrac{\pi}{3}, 2.0944/\tfrac{2\pi}{3}, 3.1416/\pi}
  \draw (-\p,\y) -- (-\p-0.15,\y) node[left=3pt]{$\lab$};



\fill[gray!30,thick] (-\p,\p-\e) rectangle (-\a,\p);
\fill[gray!30,thick] (-\p,-\p + \e) rectangle (-\a,-\p);
\fill[gray!30,thick] (\p,\p-\e) rectangle (\a,\p);
\fill[gray!30,thick] (\p,-\p + \e) rectangle (\a,-\p);
\fill[gray!30,thick] (-\b,-\e) rectangle (\b,\e);

\fill[gray!30,thick] (-\e,-\p) rectangle (\e,\p);

\coordinate (P1) at (-.5,-\p);
  \coordinate (P2) at (-3,0);
  \coordinate (P3) at (-.5,\p);

  \def\a{1}  
  \def\b{1}  
  \def\c{1}  

  \coordinate (C1) at ($(P1) + (0,\a)$);   
  \coordinate (C2) at ($(P2) + (0,\b)$);   
  \coordinate (D1) at ($(P2) + (0,-\b)$);  
  \coordinate (D2) at ($(P3) + (0,-\c)$);   

  \draw[red,thick] (P1) .. controls (C1) and (C2) .. (P2) .. controls (D1) and (D2) .. (P3); 
\end{tikzpicture}
    \caption{A non simply connnected spectrum yielding rigidity from $ \mathbb{Z} _{  + }^{ d}$.  {\it Red:} A loop non-reducible to a point.}
    \label{fig:counter-ex}
\end{figure}

Identify sets with their indicator functions and define the set / indicator function $$ S(u,v) =  \mathbb  T  ^{ 0}(u)\times 
\begin{cases} 
 \mathbb  T  ^{ 0}(v)$  if $u\in J_{0}\\
 \mathbb  T  ^{ \pi}(v)$  if $ u\in J_{\pi }\\
 \mathbb  T  ^{ 0}(v) \mathbb  T  ^{ \pi}(v)$  on $J^{ 0,\pi }.
 \end{cases}$$

The set $ S  $ is connected but not simply connected because the red loop is not contractible to a point (Figure  \ref{fig:counter-ex}). Remark that since there is no horizontal or vertical bar in the spectrum,   Proposition \ref{prop:separable} does not apply.
%

\begin{proposition}
We have for any continuous functions $\gamma _{ 1},\gamma _{ 2} $, 
\begin{align*}
 \inf_{ H\text{\rm{ polynomial}}}\| \gamma _{ 1}(u)\gamma _{ 2}(v)-H(u,v)\|_{  S} = 0.
\end{align*}
\end{proposition}
Runge's theorem says that  for any $u\in  \mathbb  T ,\gamma :  \mathbb  T  \to  \mathbb C  $ continuous and $ \varepsilon >0$ there is $ h$ polynomial such that $ \|\gamma -h\|_{  \mathbb  T  ^{ u}}<\varepsilon .$
%
Let $ h_{ 1}$ a polynomial approximation of $ \gamma _{ 1}$ on $  \mathbb  T  ^{ 0}.$
Let $ h_{ 2}^{ 0}, h_{ 2}^{ \pi }$  polynomial $ \varepsilon $-approximations of $ \gamma _{  2}$ on resp. $  \mathbb  T  ^{ 0},  \mathbb  T  ^{ \pi}$. 
To patch these  functions, introduce a trigonometric polynomial $ \varphi _{ 0}$  that is a $\varepsilon$-approximation on $  \mathbb  T    _{ 0}$  of the continuous function
\begin{align*}
 r(u): =  \begin{cases} 
 1$  on $ J_{ 0}\\
 0 $  on $J_{ \pi }\\
$linear interpolation  in between$.
  \end{cases}
\end{align*}


Remark that $\varphi _{ \pi }: =  1-\varphi _{ 0}$ is a $  \varepsilon$-approximation of $ 1-r$ on $  \mathbb  T  ^{ 0}$.
Define 
\begin{align*}
  H(u,v) = & h_{ 1}(   u)  (h_{ 2}^{ 0}(    v)\varphi _{ 0}(    \barr u) + h_{ 2}^{ \pi }(    v)\varphi _{ \pi }(   \barr  u))
\end{align*}
which is polynomial as sum and product of polynomials.
  Assume without loss of generality  $$ \|\gamma _{ i}\|\leqslant  \frac{ 1}{ 4}, \|\varphi _{ 0}\|_{  \mathbb  T  ^{ 0}},\|\varphi _{ \pi }\|_{   \mathbb  T  ^{ 0}}<  2     , \|h_{ 1}\|_{   \mathbb  T  ^{ 0}}, \|h_{ 2}^{ x}\|_{   \mathbb  T  ^{ x}}<1/2$$  hence each term of the product is bounded by $ 1.$

 We finally have on $ S$
\begin{align*}
  | H(u,v)-&\gamma _{ 1}(u)\gamma _{ 2}(v) |\leqslant  \underbrace{ | \gamma _{ 1}(u)-h_{ 1}(u) |  \mathbb  T  ^{ 0}(u)}_{\leqslant \varepsilon} \times 1+ 1\times B_{ 2}(v) \\
\text{\rm{where }}B_{ 2}(v)  : =  &
  \begin{cases} 
  2\times  | h_{ 2}^{ 0}(v)-\gamma _{ 2}(v) |  \mathbb  T  ^{ 0}(v) +\|\varphi _{ \pi }\|_{ J_{ 0}}\|h_{2}^{ \pi }\|  $  if $u\in J_{ 0}\\
     2\times | h_{ 2}^{ \pi }(v)-\gamma _{ 2}(v) |  \mathbb  T  ^{ \pi}(v) +\|\varphi _{0 }\|_{ J_{ \pi }}\|h_{2}^{ 0}\|$  if $u\in J_{ \pi }\\
      | \varphi_{ 0} (u)h_{ 2}^{ 0}(v) + (1-\varphi _{ 0}(u))h_{ 2}^{ \pi }(v) -\gamma _{ 2}(v)|  \mathbb  T  ^{ 0}(v)  \mathbb  T  ^{ \pi}(v)$  otherwise$.
   \end{cases}
\end{align*}
 
Since $ \|\varphi _{ 0}\|_{ J_{ \pi }},\|\varphi _{ \pi }\|_{ J_{ 0}}\leqslant \varepsilon$ by construction, the   first and second cases are bounded by $ 2\varepsilon + \varepsilon $. Using $ \gamma _{ 2} = (\varphi _{ 0} + (1-\varphi_{ 0})) \gamma _{ 2},$  the third term is bounded on $  u\in \mathbb  T  ^{ 0}$ by 
\begin{align*} 
\varphi_{ 0} (u)    | h_{ 2}^{ 0}(v)-\gamma _{ 2}(v) |   \mathbb  T  ^{ 0}(v) +  (  1-\varphi_{ 0} (u))   | h_{ 2}^{ \pi }(v)- \gamma _{ 2}(v) |   \mathbb  T  ^{ \pi}(v)
<2\varepsilon.\end{align*}
Adding up the error terms gives a bound in $ 4\varepsilon $, this concludes the proof.

 \section{Rigidity phase transition for finite-range random fields}
 \label{sec:finite-range} 
 We now exhibit a distinct mechanism leading to maximal rigidity in a continuous, non-hyperuniform random field with finite dependence range.    
 Let $ \Delta  = 1_{ B(0,1)}  \ast 1_{ B(0,1)} $ the autocorrelation of the unit ball (called triangle covariance in dimension $ 1$), supported by $ B(0,2).$ For some $ f \in \C_{ c}^{ b}(\mathbb{R}^{ d}),$ say that $ \M$ is linearly $f$-rigid   if for some $ h_{ n}\in \C_{ c}^{ b}(  { \rm supp}(f)^{ c}   )$, $n\geqslant 1$, $\mathbf E  | \M(f )-\M(h_{ n}) | ^{ 2}\to 0.$
  {  Recall that $ \Leb$ is Lebesgue measure.}
  \begin{theorem} 
 \label{thm:short-range} 
 Let $ \M $ be a \WSS random measure on $ \mathbb{R}^{ d}$ with covariance 
$
\C= \Delta \Leb$. Then $ \M$ is LMR on $ B(0,\rho )$ if $ \rho <\frac{2}{\pi }$. On the other hand, for $ \rho >2$,  $ \M$ is not $f $-rigid for $ f $ with $  { \rm supp}(f ) = B(0,\rho )   $.\end{theorem}    
  Hence the field is rigid on small balls but becomes uncorrelated beyond range 2, revealing a genuine rigidity phase transition.
  Let us explicit what it  means for a continuous Gaussian field:
  
  \begin{example}
  \label{ex:gauss-rigid}
  Let $ \X$ the unique centred Gaussian random field with covariance function 
\begin{align*}
\mathbf{E}(\X(x)\X(y)) = \Delta (x-y);x,y\in \mathbb{R}^{d}.
\end{align*}
By standard results, $ \X$ has continuous sample paths (see  \cite{AT07}).
Then the previous conclusion holds: for $ y_{0}\in B(0,2 /\pi ),\varepsilon >0$ there are deterministic points $ x_{1},\dots ,x_{m}\in B(0,2/\pi )^{c}$ and coefficients $ h_{1},\dots ,h_{m}\in  \mathbb R $ with 
\begin{align*}
\mathbf{E}\left|\X(y_{0})-\sum_{k}h_{k}\X(x_{k})
\right|^{2}<\varepsilon .
\end{align*} 
On the other hand,   the $ \X(x ),\|x\|>2$ are independent of $ \X(0).$

  \end{example} 
  
From a modelling perspective, this field is one of the most basic example of a stationary continuous  field, both conceptually and on the practical aspect of simulation or estimation.
 To the author's knowledge, this extremely rigid behavior has not been observed anywhere, even in dimension $ 1$, for such finite range models, although it could likely be useful for reconstruction in signal processing. In dimension $ d = 1, \Delta (x) = (1-|x|/2)1_{ |x|\leqslant 2}$, but the expression is less straightforward in higher dimensions. Similar results hold for simpler continuous functions, such as $ (1-\|x\|)1_{ \|x\|\leqslant 1}$.
Replacing $ \Delta  $ with $ \Delta^{\ast   k} $ for $ k\in \mathbb{N}$, one can also derive similar examples with arbitrarily high regularity. 

The coefficients $ h_{ k}$ in Example  \ref{ex:gauss-rigid} can be  computed with classical Fourier analysis in $ L^{ 2}( \hat \Delta )$ after applying \eqref{eq:phase} to approximations of Dirac masses in $ y_{ 0}$ and the $ x_{ k}'s$, see Section \ref{sec:reconstruction}. Note that some $ x_{ i}'s$ will likely have to lie in $ B(y_{ 0},2)^{ c}$, even though the $ \X(x_{ i})$ are independent from $\X(y_{ 0}) $, to compensate for the potential error induced by the $ x_{ i}$'s in $ B(y_{ 0},2).$

It is worth stressing that the observed rigidity does not stem from finite-range dependence, from smoothness of the field, or from integrability of the covariance.
\begin{proposition}
Let  $\X$ be a random field with covariance
 $\Delta  (x)^{q}$ for some $ q\geqslant 2$.  Then for $ \varepsilon >0,$ $\X$ is not linearly $ 1_{ B(0,\varepsilon )}$-rigid   (and a fortiori not LMR).
\end{proposition}

\begin{proof} Let $ J_{d}$ denote the Fourier transform of  the unit ball indicator. It is given by
\begin{align}
\label{def:Jd}
J_{d} (u):= \frac{ B_{d/2}( \|u\| )}{\|u\|^{d/2}},u\in \mathbb{R},
\end{align}
where $ B_{d/2}$ is the Bessel function of the first kind of order $ d/2$; in particular $ J_{ 1} =2 \text{\rm{\color{black} sinc}}$. Also, $ J_{ d}$ is entire as the Fourier transform of a bounded function with compact support.  We have   as $ u\to \infty $  the classical asymptotics 
\begin{align}
\label{eq:bd-bessel}
J_{d}(u) = c' \|u\| ^{-(d + 1)/2}\cos( \|u\|-c_{d})(1 +O(\|u\|^{-1}) ).
\end{align}The spectral density is  $$\s: =  \widehat{ \Delta ^{ q}} =  (J_{ d}^{ 2})^{  \ast  q}: =  J_{ d}^{ 2} \ast  \dots  \ast  J_{ d}^{ 2}.$$  Since $ J_{ d}$ is continuous with $ J_{ d}(0)>0$, we have $ J_{ d}^{ 2}\geqslant c1_{ [-c,c]^{ d}}$  for some $ c>0$ hence $(J_{ d}^{ 2})^{  \ast  q}\geqslant a_{ q}1_{ [-b_{ q},b_{ q}]^{ d}} $ for some $ a_{ q},b_{ q}>0$, and, using  \eqref{eq:bd-bessel}, for some $ c_{ q}>0,$
\begin{align*}
\s(u) \geqslant a_{ q-1}J_{ d}^{ 2}  \ast  1_{ [-b_{ q-1},b_{ q-1}]^{ d}}\geqslant c _{ q}(1 + \|u\|)^{ -d-1}.
\end{align*} Define
\begin{align*}
 \psi (u) = J_{ d}(u\varepsilon /2)^{ 2}
\end{align*}
whose Fourier transform   is bounded and supported by $ B(0,\varepsilon ).$
Since $ \psi (u)^{ 2} = O(\|u\|^{ -2(d + 1)}),$
\begin{align*}
 \int_{} \psi ^{ 2} \s^{ -1}<\infty .
\end{align*}
Hence $ \varphi  = \psi /\s\in L^{ 2}(\s)$ with $ \sp(\varphi \s) = \sp(\psi )\subset B(0,\varepsilon )$, which by Theorem \ref{thm:general-rigidity}
implies no LMR on $ B(0,\varepsilon ).$ The non-$ 1_{ B(0,\varepsilon  )}$-rigidity is a consequence of  \cite[Theorem 7]{Lr24}, using $ \psi (0)\neq 0$.
\end{proof}

On the other hand, Theorem  \ref{thm:short-range} remains valid if the covariance $ \Delta   $ is convolved with another covariance $ \C'$, since a convolution can only increase the zero set of the Fourier transform. 

\begin{remark}

Fields which are a.s. analytic, like the Bargman-Fock field, are also rigid but for  analycity reasons. This is a different phenomenon, they are in particular rigid on any set whose complement has a non-empty interior, and there is no phase transition.

More generally, one can show with Theorem \ref{thm:stealthy-cone} that if the spectral measure vanishes on a cone with non-empty interior, then the corresponding random measure is perfectly interpolable  \underline {from} any non-empty open set, reminiscent of the behaviour of random fields which are a.s. entire functions.  
\end{remark}

  \subsection{Proof of Theorem  \ref{thm:short-range}}  
 Define $ \E_{ d}(R)$ as the class of entire functions $ \psi $ of $ d$ variables of  type $ R$, i.e. such that  for some $ C,N,$
\begin{align*}
  | \psi (z) | \leqslant C \| z \| ^{ N}\exp(R \| z \| ),z\in  \mathbb C ^{ d}.
\end{align*}

 {It is a classical consequence of Jensen's identity in complex analysis   \cite[Chap.1]{Koosis} and the Schwartz'Paley-Wiener theorem (\cite[Th. 7.22]{Rudin2})  that the constant $ R$ is related to the density of zeros of a radial function.    For $ \psi :  \mathbb C  \to  \mathbb C $, define $ n_{ T}(\psi ) = \#\{z\in B_{  \mathbb C }(0,T):\psi (z) = 0\}$ and
\begin{align*}
 \z(\psi ) = \liminf_{ T\to \infty }\frac{n_{ T}(\psi )}{T }.
\end{align*}}

           \begin{lemma}[Jensen]
 \label{lm:jensen}
Let $ \psi \in \E_{ 1}(R)  \setminus \{0\}$. Then   $\z(\psi ) \leqslant R.$
 \end{lemma}

 This lemma retrieves a classical example of uniqueness pair: for any $ R>0,z<R^{ -1}$, $ A = B(0,R)$ forms a uniqueness pair with $ \hat A = \cup _{ k\in \mathbb{N}^{ *}} \partial B(0,kz)$, see Table  \ref{table:1}-(e).
   \begin{proof}[Proof of Theorem  \ref{thm:short-range}]
 
 The starting point is that the spectral density is $ \s =  J_{ d}^{ 2}$. Let the radial part $  \mathscr  j(\|u\|) =J_{ d}(u) $. From  \eqref{eq:bd-bessel} we have
\begin{align*}
\z(  \mathscr  j  )=\z(\cos )=\frac{2}{\pi }.
\end{align*}
We prove by contradiction that there is LMR on $ A = B(0,2/\pi-\eta  )$ for $ \eta >0$. To exploit Theorem \ref{thm:general-rigidity}, we therefore assume that there is $ \varphi \in L^{ 2}(\s)  \setminus \{0\}$ with $ \sp(\varphi \s)\subset B(0,2/\pi -\eta )$.   Let $ \psi  = \varphi \s.$
We have  $ \psi \in L^{ 2}(\s^{ -1})$ because 
\begin{align*}
 \int_{}\psi ^{ 2}/\s = \int_{}\varphi ^{ 2}\s<\infty ,
\end{align*}
and $ \psi \in \E_{ d}( 2/\pi -\eta /2)$ by assumption.
An important ingredient is  \cite[Lemma 7]{Lr24} that  {  we formulate in the current context:

\begin{lemma}
Let $ \M$ a weakly stationary measure whose spectral density $ \s$ is invariant under rotations. If there exists $ \psi \in L^{ 2}(\s^{ -1})\cap \E(B(0,\pi /2-\eta ))  \setminus {\{0\}}$, there exists a non-null function $ \psi _{ 0}\in L^{ 2}(\s^{ -1})\cap \E(B(0,2/\pi -\eta /2))$ that is also invariant under rotations.
\end{lemma}
}Hence  $ \psi _{ 0}(u) =  \tilde \psi (\|u\|)$ for some function $ \tilde \psi $ on $ \mathbb{R}_{  + }$. Since $ \psi_{ 0} $ is entire in $ d$ variables,  $ \tilde \psi$ is entire, and furthermore $ \tilde \psi $ is of exponential type $ 2/\pi  - \eta /2 $:
\begin{align*}
 |  \tilde \psi (\|u\|) |  =  | \psi (u) | \leqslant C\exp( (2/\pi  - \eta /2)  \|u\|).
\end{align*}
By Lemma \ref{lm:jensen},
\begin{align*}
 \z( \tilde \psi )\leqslant \frac{2}{\pi }  -\eta /2  .
\end{align*}
This actually is in contradiction with the fact that $ \tilde \psi $ vanishes on each $ 0$ of $  \mathscr  j$ (because $ \psi _{ 0}\in L^{ 2}(\s^{ -1})$), which yields 
\begin{align*}
 \z( \tilde \psi )\geqslant  \frac{2}{\pi }.
\end{align*}
This concludes the proof of LMR. 
   
    {\bf Non-rigidity:} The non-rigidity is pretty obvious given the fact that, due to the finite range, for $ \varepsilon >0,f \in \mathscr C_{c}^{ b} (B(0, \varepsilon ))$, $ \X(f )$ is decorrelated from $ \X (h),h\in \mathscr C_{c}^{ b}(B(0,2 + \varepsilon )^{ c})$:  \eqref{eq:cov-measure} yields
\begin{align*}
 \textrm{Cov}\left(\X(f ),\X (h)\right)= \int_{}\mathbf{1}_{\{\|x\|\leqslant 2,\|x+y\|\geqslant 2+\varepsilon ,\|y\|\leqslant \varepsilon \}}f (y)h(x+y)\C(x)dxdy=0.
\end{align*}
Any estimator of the form $ \hat \X(f )=\X (h)$ satisfies 
\begin{align*}
 \textrm{Var}\left(\X(f )- \hat \X(f ) \right) \geqslant  \textrm{Var}\left(\X(f )\right)
\end{align*}
 which makes it impossible for the error to approach $ 0.$
   \end{proof}
 For this phenomenon to occur, the zeros of the spectral measure need have a positive density. For continuous point processes, the only processes we are aware of as having such a behaviour are  lattices independently perturbed by i.i.d.   variables having a discrete support, which are in fact very close to the class of discrete fields.

%
%
%
 
\section{Appendix: Visual representation of results} 
\label{app}

See next page.

 \section*{Data availability statement}  
 This study did not generate or analyse any datasets, as it is based on theoretical analysis and/or existing literature. Therefore, data sharing is not applicable.      \\

  \section*{Conflict of interest}   
 The authors declare that they have no known competing financial interests or personal relationships that could have appeared to influence the work reported in this paper.
     \pagestyle{empty}
 

\vspace{-1cm} 

\begin{table}
\begin{tabular}{  | m{.5cm}| m{4cm}  | m{6cm} |  m{5cm} | }
\hline
&Spectrum & Interpolation & Direct space \\
\hline
a&\begin{tikzpicture}[scale = .6]

\draw[black, line width=1.5mm] (-2,0) -- (0.8,0);   
\draw[black, line width=1.5mm] (1,0) -- (3,0);    

\end{tikzpicture}

1D gap
& \begin{tikzpicture}[scale = .5]


\fill[gray!20] (-5,-0.1) rectangle (0,0.1);
\draw[->] (-5,0) -- (5,0);

\draw (0,0.1) -- (0,-0.1);
\node at (0,-0.5) {0};

\foreach \x in {-4.5,-3.5,-3.2,-2.6,-2.3,-2.0,-1.4,-1.1,-0.5,-0.2}
    \fill (\x,0) circle (2pt);

\draw[->, thick, bend left=45] (-2,0.4) to (2,.5);

\end{tikzpicture}
 & 
$\mathbb{R}$ or $ \mathbb{Z} $ : Krein theoreom, Theorem \ref{thm:SW} \\
\hline
b&
\includegraphics[width = 3cm]{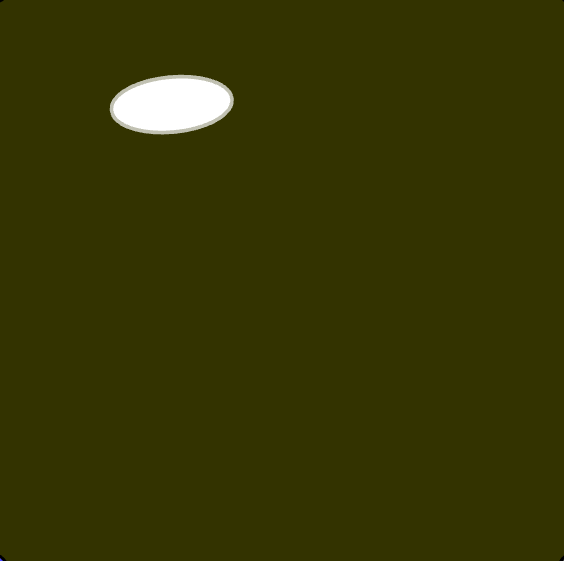}

Spectral hole

 \begin{tikzpicture}[scale=0.5, font=\small]

\def\p{3.1416}      
\def\a{1.0472}      
\def\b{2.0944}      
\def\e{.1}
\fill[gray!10,draw=white] (-\p,-\p) rectangle (\p,\p);

\fill[black,draw=white] (-\p,-\e) rectangle (\p,\e);


\end{tikzpicture}

&\begin{tikzpicture}

\def\thetaMin{-100}   
\def\thetaMax{100}    
\def\rMin{0.5}         
\def\rMax{3}         
\def\nPoints{50}     

\fill[gray!20] (0,0) -- (\thetaMin:\rMax) arc (\thetaMin:\thetaMax:\rMax) -- cycle;

\draw[gray, thick] (0,0) -- (\thetaMin:\rMax);
\draw[gray, thick] (0,0) -- (\thetaMax:\rMax);

\foreach \i in {1,...,\nPoints}{
    \pgfmathsetmacro{\r}{rnd*(\rMax-\rMin)+\rMin}
    \pgfmathsetmacro{\theta}{rnd*(\thetaMin-\thetaMax)+\thetaMax} 
    \fill (\theta:\r) circle (.8pt);
}

\draw[->, thick, bend right=45] (1,1) to (-2,1);
\end{tikzpicture}

& $ \mathbb{R}^{ d}:$ Stealthy processes, Theorem \ref{thm:stealthy-cone}\\
\hline
c&
        \includegraphics[width=3cm]{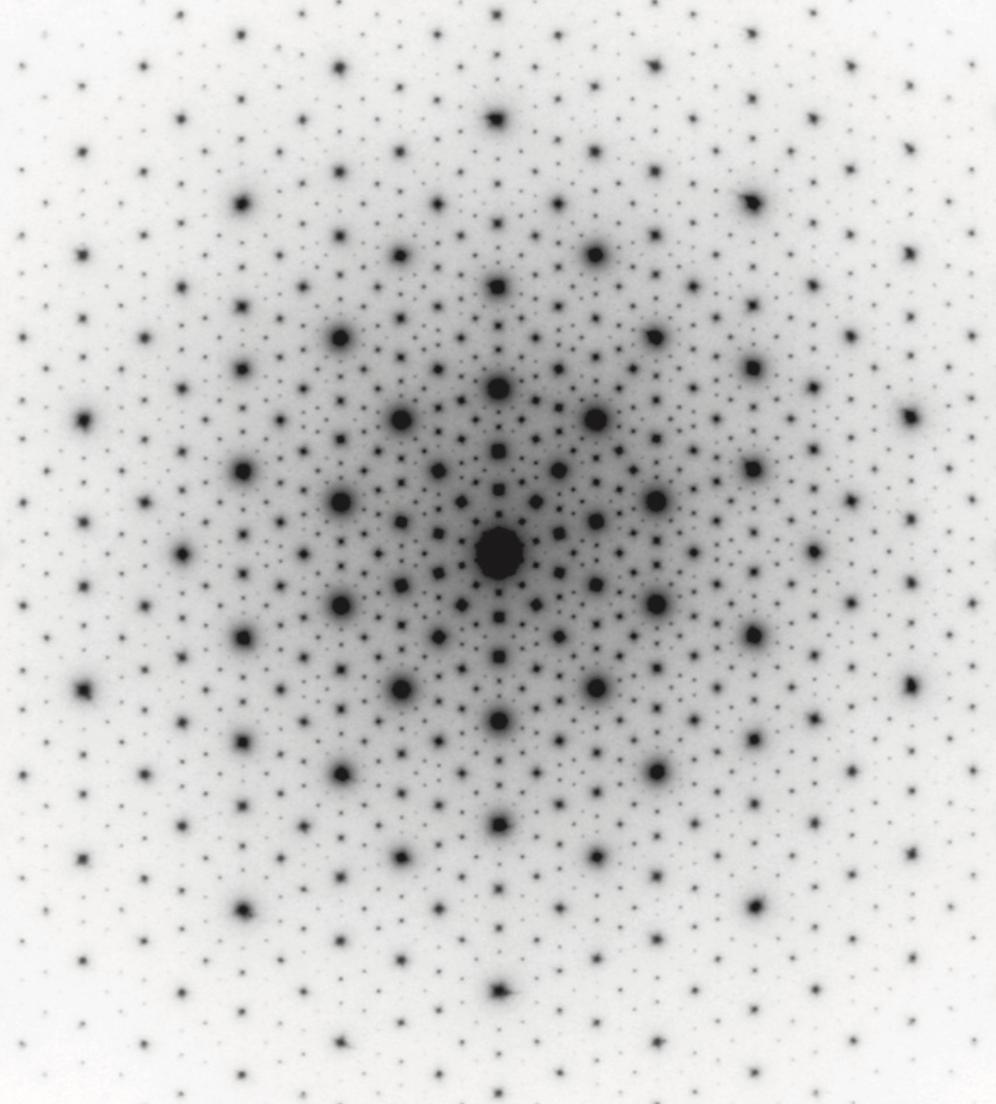}
        Purely atomic \tiny \copyright   Conradin Beeli,  \cite{baake-book}.    &

\begin{tikzpicture}

\def\thetaMin{-10}   
\def\thetaMax{10}    
\def\rMin{0.5}         
\def\rMax{3}         
\def\nPoints{20}     

\fill[gray!20] (0,0) -- (\thetaMin:\rMax) arc (\thetaMin:\thetaMax:\rMax) -- cycle;

\draw[gray, thick] (0,0) -- (\thetaMin:\rMax);
\draw[gray, thick] (0,0) -- (\thetaMax:\rMax);

\foreach \i in {1,...,\nPoints}{
    \pgfmathsetmacro{\r}{rnd*(\rMax-\rMin)+\rMin}
    \pgfmathsetmacro{\theta}{rnd*(\thetaMin-\thetaMax)+\thetaMax} 
    \fill (\theta:\r) circle (.8pt);
}

\draw[->, thick, bend right=45] (1,0) to (-2,0);
\end{tikzpicture}

&$ \mathbb{R}^{ d}$: Quasicrystals, Theorem \ref{thm:quasi}, \\
\hline
d&

\begin{tikzpicture}[scale=1.5]
\fill[black,draw=white] (-1,-1) rectangle (1,1);


\draw[white, very thick, domain=-1:1, samples=10, smooth, variable=\x]
plot (\x,{0.75*(1 - \x*\x)^5 + .1});

\draw[white, very thick, domain=-1:1, samples=10, smooth, variable=\y]
plot ({-0.75*(1 - \y*\y)^4 + .6},-\y);
\end{tikzpicture}

Simply connected  ($  \mathbb  T  ^{ d}$)

&
\hspace{1cm}  
\begin{tikzpicture}[scale=0.1]
  \def\N{20} 
  \foreach \i in {0,...,19}{
    \foreach \j in {0,...,19}{
      \pgfmathparse{int(random(0,1))} 
      \ifnum\pgfmathresult=1
        \fill[black] (\i,\j) rectangle ++(1,1);
      \else
        \fill[white] (\i,\j) rectangle ++(1,1);
      \fi
    }
  }

  \draw[red!80!black,thin] (0,0) rectangle (\N,\N);
  \foreach \k in {1,...,19} {
    \draw[red!80!black,thin] (\k,0) -- (\k,\N);
    \draw[red!80!black,thin] (0,\k) -- (\N,\k);
  }

\draw[->, thick, bend right=25] (5,8) to (-5,8);
\draw[->, thick, bend right=25] (16,8) to (26,8);
\draw[->, thick, bend right=25] (12,3) to (12,-5);
\draw[->, thick, bend right=25] (12,18) to (12,25);
\end{tikzpicture} & 

$ \mathbb{Z} ^{ d}: $ Periodic fields, Theorem \ref{thm:BSW-d}
\\
\hline
e&
 \begin{tikzpicture}[scale=0.5, font=\small]

\def\p{3.1416}      
\def\a{1.0472}      
\def\b{2.0944}      
\def\e{.3}
\fill[black,draw=white] (-\p,-\p) rectangle (\p,\p);




\def\gr{gray!10}
\fill[\gr,thick] (-\p,\p-\e) rectangle (-\a,\p);
\fill[\gr,thick] (-\p,-\p + \e) rectangle (-\a,-\p);
\fill[\gr,thick] (\p,\p-\e) rectangle (\a,\p);
\fill[\gr,thick] (\p,-\p + \e) rectangle (\a,-\p);
\fill[\gr,thick] (-\b,-\e) rectangle (\b,\e);

\fill[\gr,thick] (-\e,-\p) rectangle (\e,\p);

\end{tikzpicture}

Not simply connected counter-example
&
\begin{tikzpicture}

\def\thetaMin{0}   
\def\thetaMax{90}    
\def\rMin{0.5}         
\def\rMax{3}         
\def\nPoints{20}     

\fill[gray!20] (0,0) -- (\thetaMin:\rMax) arc (\thetaMin:\thetaMax:\rMax) -- cycle;

\draw[gray, thick] (0,0) -- (\thetaMin:\rMax);
\draw[gray, thick] (0,0) -- (\thetaMax:\rMax);

\foreach \i in {1,...,\nPoints}{
    \pgfmathsetmacro{\r}{rnd*(\rMax-\rMin)+\rMin}
    \pgfmathsetmacro{\theta}{rnd*(\thetaMin-\thetaMax)+\thetaMax} 
    \fill (\theta:\r) circle (.8pt);
}

\draw[->, thick, bend right=45] (1,1) to (-1,1);
\end{tikzpicture}

&$ \mathbb{R}^{ 2}$: Counter-example, Section \ref{sec:counter-ex-SC}\\
\hline
f&
\includegraphics[width = 3cm]{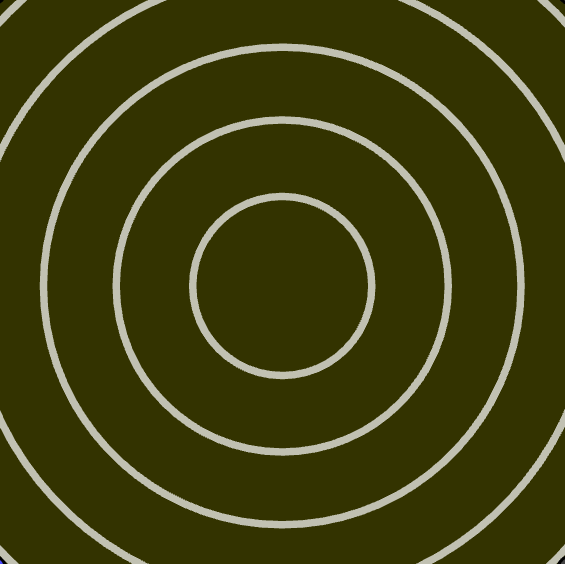}

Concentric circles
&
\hspace{.3cm}  
\includegraphics[width = 5cm]{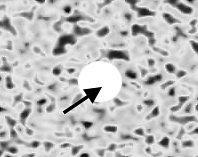}
&$\mathbb{R}^{ d}$: Triangle Gaussian field, Theorem \ref{thm:short-range}\\
\hline
\end{tabular}
\caption{Uniqueness pairs and interpolation}
\label{table:1}
\end{table}
~
\newcommand{\etalchar}[1]{$^{#1}$}


\end{document}